\documentclass[a4paper,11pt]{amsart}

\usepackage{pgfplots}
\pgfplotsset{compat=1.15}
\usepackage{mathrsfs}
\usetikzlibrary{arrows}
 
\usepackage[utf8]{inputenc}
\DeclareUnicodeCharacter{00A0}{ }
\usepackage{amssymb}
\usepackage{MnSymbol}
\usepackage{enumitem}

\usepackage[normalem]{ulem}

\usepackage{tcolorbox}

\usepackage{ae}
\usepackage[utf8]{inputenc}
\numberwithin{figure}{section}

\usepackage{mathrsfs,a4wide}
\usepackage{esint}
\usepackage{color}

\usepackage{comment}

\usepackage[colorlinks=true, linkcolor=red, urlcolor=red,citecolor=red]{hyperref}

\usepackage{import}

\usepackage[leqno]{amsmath}

\numberwithin{figure}{section}

\newtheorem{theorem}{Theorem}[section]
\newtheorem*{mtheorem}{Main Theorem}
\newtheorem{lemma}[theorem]{Lemma}
\newtheorem{proposition}[theorem]{Proposition}

\theoremstyle{definition}

\newtheorem{remark}[theorem]{Remark}
\numberwithin{equation}{section}

\newcommand{\R}{\mathbb{R}}
\newcommand{\N}{\mathbb{N}}

\newcommand{\F}{\mathcal {F}}
\renewcommand{\H}{\mathcal{H}}
\newcommand{\Si}{\mathbb{S}}
\newcommand{\beq}{\begin{equation}}
\newcommand{\eeq}{\end{equation}}

\newcommand{\eps}{\varepsilon}
 
\newcommand{\la}{\langle}
\newcommand{\ra}{\rangle}

\newcommand{\pa}{\partial}

\newcommand{\medint}{-\kern -,375cm\int}
\newcommand{\medintinrigo}{-\kern -,315cm\int}

\begin{document}

\title[Fractional mean curvature  for convex sets]{Convergence of the volume preserving fractional mean curvature flow  for convex sets}

\author{Vesa Julin}

\author{Domenico Angelo La Manna}


\begin{abstract} 
We prove that   the volume preserving fractional mean curvature flow starting from a convex set does not develop singularities along the flow. By the recent result of Cesaroni-Novaga \cite{CN} this then implies that the flow converges to a ball exponentially fast. In the proof we show that the a priori estimates due to  Cinti-Sinestrari-Valdinoci \cite{CSV2} imply the $C^{1+\alpha}$-regularity of the flow and then provide a regularity argument which improves this to  $C^{2+\alpha}$-regularity of the flow. The regularity step from $C^{1+\alpha}$ into $C^{2+\alpha}$ does not rely on convexity and can be adopted to more general setting.    
\end{abstract}

\maketitle

\section{Introduction}

We say that a given set $E_0 \subset \R^{n+1}$ evolves under the volume preserving fractional mean curvature flow if there exists a flow $(E_t)_{t\in [0,T)}$ of sets starting from $E_0$, which satisfies the equation 
\begin{equation}
\label{eq:the-flow}
V_t = -(H_{E_t}^s - \bar H_{E_t}^s) \qquad \text{on } \, \partial E_t,
\end{equation}
where $V_t$ is the normal velocity, $H_E^s$ with $s \in (0,1)$ is the fractional mean curvature defined as 
\begin{equation}
\label{def:fracmean}
H_E^s(x) = \int_{\R^{n+1} \setminus E} \frac{dy}{|y-x|^{n+1+s}} - \int_{E} \frac{dy}{|y-x|^{n+1+s}} 
\end{equation}
and $\bar H_{E_t}^s$ is its integral average. The flow \eqref{eq:the-flow} bears many similarities to the classical volume preserving mean curvature flow, in particular,  it can be seen as   a gradient flow of the fractional perimeter. The flow \eqref{eq:the-flow} can also be seen as  a perturbation of the fractional mean curvature flow, with the difference that \eqref{eq:the-flow} preserves the volume and can thus be interpreted as the evolutionary counterpart of the fractional isoperimetric inequality.  

The study of  the fractional perimeter problem goes back to \cite{CRS}, where the authors prove so called $\eps$-regularity result, i.e.,  partial regularity of the perimeter minimizers. As with minimal surfaces, also in the fractional case the perimeter minimizers may have singularities in higher dimensions. The question of optimal regularity has been studied extensively and we merely refer to \cite{SV}, which is the best known result on the size of the singular set at the moment. We also refer to \cite{BFV}, where the authors provide a bootstrap argument  to obtain higher order regularity outside the singular set. The related isoperimetric problem, which states that  the ball minimizes the fractional perimeter under the volume constraint, follows from standard symmetrization argument and the sharp quantitative version is proven in  \cite{FFMMM}  (see also \cite{FMM}). Moreover, the fractional version of the Alexandrov theorem, which states that  the only regular set with constant fractional mean curvature is the ball, is proven independently in \cite{CFSW, CFMN}. This is important as it essentially  implies that, if the flow  \eqref{eq:the-flow} is well defined for all times, then it converges to the ball. 

In \cite{JL} we  prove the short time existence of the classical solution of \eqref{eq:the-flow} for $C^{1,1}$-regular initial sets. This result also applies to the fractional mean curvature flow, in which case the existence of the level set solution is proven in \cite{imb} and another weak solution is constructed in \cite{CMP}.  As usual with geometric flows,  the fractional mean curvature flow may develop singularities in finite time such as neckpinching \cite{CSV1} and fattening \cite{CDNV}. However, at the moment there is no systematic classification of the possible types of singularities in the fractional setting similar to the classical mean curvature flow \cite{MantegazzaBook}.  We are also not aware of a construction of a weak solution for \eqref{eq:the-flow}.  


A natural problem is whether we may avoid the formation of singularities, if we constraint ourselves to  special type of  initial sets.  In \cite{CN} the authors prove, that if the initial set is close enough to the ball, then the flow \eqref{eq:the-flow} remains smooth and converges exponentially fast to a translation of the  ball. Another special case is when the initial set is a subgraph of a Lipschitz function. Note that then the set is unbounded and therefore, instead of \eqref{eq:the-flow}, it is natural to study the standard fractional mean curvature flow. In this setting the authors   \cite{SaeV} prove  important a priori estimates, which essentially imply that if the flow starts from Lipschitz graph with bounded fractional curvature, these quantities are preserved along the flow.  This problem is studied further in \cite{CN0}, where the authors show that in the above setting the fractional mean curvature flow  does not develop singularities and converges to the flat surface exponentially fast. This is the fractional counterpart of the Ecker-Huisken result \cite{EH} for the mean curvature flow with the only difference that the initial set is assumed to have bounded fractional curvature. Due to the results \cite{MW, Ser} one might expect the result to hold without the curvature assumption, but to the best of our knowledge this is not known at the moment. 
   
As we already mentioned, here we consider the case of convex initial sets, which are bounded and $C^{1,1}$-regular. In the case of the classical volume preserving mean curvature flow the result by  Huisken \cite{Hui2}  implies that the flow remains smooth and converges expontially fast to the ball. In the fractional setting,  this result is proven in   \cite{CN} and \cite{CNR} under the additional assumption that the flow remains $C^{2 +\alpha}$-regular. Here we show that this additional regularity assumption is conditional, or in other words, this regularity assumption can proven  for convex initial sets, and thus we obtain the Huisken result in the fractional setting.

\begin{mtheorem} \label{main thm}
Let $E_0 \subset \R^{n+1}$ be $C^{1,1}$-regular convex set with $|E_0| = |B_1|$. Then the classical solution of \eqref{eq:the-flow} exists for all times $(E_t)_{t\geq 0}$ and converges to the ball $B_1(x_0)$, for some $x_0 \subset \R^{n+1}$, in $C^k$ for every $k \in \N$  exponentially fast. To be more precise, there is $T \geq 1$ depending on the initial set $E_0$ and on the dimension,  such that we may write the moving boundaries by the height functions over the unit sphere $h(\cdot,t): \Si^n \to \R$ as 
\[
\pa E_t = \{ h(x,t)x + x_0 : x \in \Si^n\}
\]
for $x_0 \in \R^{n+1}$ when $t \geq T$ and  it holds $\|h(\cdot,t) -1 \|_{C^{k}( \Si^n)} \leq C_{k} e^{ - c_{k}  t}$,  
where the constants $c_k>0$ and $C_k \geq 1$ depend on $n, k, T$ and on  the initial set $E_0$. 

\end{mtheorem}

Let us comment the proof of the Main Theorem. First, we remark that the $C^{1,1}$ assumption on $E_0$ is needed only to guarantee the existence of the flow, as our proof relies on the short time existence result in \cite{JL}.
As we mentioned above, the proof relies on  the a priori estimates proven in \cite{CN, CNR, CSV2}. Indeed, by the result in \cite{CNR} the sets $E_t$ along the flow remain convex, while the estimates in \cite{CSV2} imply that their fractional mean curvature remains uniformly bounded, i.e. 
\[
\sup_{t \in [0,T)} H_{E_t}^s(x) \leq C
\]
and there are radii $0<r<1<R$ such that $B_r(x_t) \subset E_t \subset B_R(x_t)$ as long as the flow is classically defined. These important estimates are our starting point and we first prove in Proposition \ref{prop.curvaturebound} that for convex sets the above fractional  curvature bound implies uniform $C^{1+s}$-H\"older regularity. We then recall that the authors in \cite{CN} (see also \cite{CSV2}) show that, if we would be able to improve the $C^{1+s}$-regularity into uniform $C^{2+\alpha}$-regularity, then the flow \eqref{eq:the-flow} does not develop singularities and converges exponentially fast to a ball. Our main contribution here is to prove this regularity step.
Roughly speaking, the idea is that geometric flows do not develop singularities under uniform $C^{1+\alpha}$-regularity estimate,  unless the singularity is due to change of topology. This is due to the fact that one may parametrize the equation such that it becomes a quasilinear PDE. We  use this idea in \cite{JL2} in a more complicated equation, and since we know that in our case the flow \eqref{eq:the-flow}  remains convex, and thus the topology does not change, we may adopt it in this setting.   

The first technical issue in the proof is to find a good parametrization of the flow. The idea is to differentiate the equation \eqref{eq:the-flow} with respect to space and parametrize the equation
\[
\nabla_{X(t)} V_t = - \nabla_{X(t)} H_{E_t^s}  \qquad \text{on } \, \pa E_t,
\]
where $X(t)$ is a vector field on the moving boundary $\pa E_t$. By doing this we obtain a quasilinear PDE for the derivative of the height function, where all the error terms are lower order. We point out that we could also first parametrize the original equation \eqref{eq:the-flow} by the height function and differentiate that, but this leads to much more complicated calculations. We  then improve some of the methods developed in \cite{JL} and prove the $C^{2+\alpha}$ regularity of the flow in Section 5. The main difference to \cite{JL} is that here we need sharp  estimates for the error terms in order to show that they do not blow up in finite time. Another issue is that here our equation is not a small parturbation of the fractional heat equation with respect to the sphere, and we need Schauder estimates for general linear parabolic equations defined on the unit sphere (see Section 4). Similar Schauder estimates are proven  in \cite{MW} in the case of evolving periodic graphs. 

The paper is organized as follows. In Section 2 we recall some results in the existing literature and prove the $C^{1+s}$-estimate of the flow. Section 3 is mostly devoted  to the calculations of  the parametrization of the equation. In Section 4 we generalize the Schauder estimates for general linear parabolic equations in  the Euclidean setting proven in  \cite{MP} (see Theorem \ref{teo:parabolic})  to the case of the sphere  (see Theorem \ref{thm2:schauder}). Since the result in the Euclidian setting is scattered in literature, we decide to give a self-contained proof  in the Appendix. Finally in Section 5  we prove that $C^{1+s}$-regularity of the flow implies uniform $C^{2+s+\alpha}$-regularity, which concludes the proof of the Main Theorem. 

\section{Notation and preliminary results}

Throughout the paper $C \geq 1$ and $c>0$ stand for  generic constants which may change from line to line. If needed, we specify their dependence on relevant parameters.  We denote the open ball with radius $r$ centered at $x$ by $B_r(x)$ and by $B_r$ if it is centered at the origin. The notation does not specify the dimension of the ball, but in order to avoid possible confusion, we  write $B_r(x) \subset \R^{k}$  to point out that the ball is $k$-dimensional. We denote the inner product  of $x,y \in \R^{n+1}$ by $\la x,y \ra$. 

For a given smooth and bounded set $E \subset \R^{n+1}$ the fractional mean curvature of order $s \in (0,1)$ at $x \in \pa E$ is defined as \eqref{def:fracmean}
in the principal valued sense. In order to define the normal velocity we say that a family of smooth sets $(E_{t})_{t \in [0,T)}$ is a smooth flow starting from $E_0$, if there is a family of  diffeomorphisms $\Phi_t : \R^{n+1} \to \R^{n+1}$ such that $t \mapsto \Phi_t(x)$ is smooth,  $\Phi_t(E_0) = E_t$ and $\Phi_0(x) = x$. The normal velocity at $y = \Phi_t(x) \in \pa E_t$ is then defined as $V_t(y) = \la \frac{\pa}{\pa t} \Phi_t(x), \nu_{E_t}(y) \ra$. Finally we say that a smooth flow  $(E_{t})_{t \in [0,T)}$ starting from $E_0$ is a solution to \eqref{eq:the-flow}, if it satisfies the associated equation point-wise. 

We recall that  if the initial set $E_0 \subset \R^{n+1}$ is bounded and $C^{1,1}$-regular, i.e., its boundary is $C^{1,1}$-regular hypersurface, then  by \cite{JL}  the equation \eqref{eq:the-flow}  has a solution for a short time $(E_t)_{t \in [0,T)}$  and  the sets $E_t$ for $t >0$ are smooth ($C^\infty$ -regular hypersurfaces). In this paper we restrict  to  convex initial sets $E_0$, which are $C^{1,1}$-regular. Then by the results in \cite{CNR, CSV2} we know that the sets $E_t$ remain convex, there are $r>0, R >1$ such that $B_r(x_t) \subset E_t \subset B_R(x_t)$ for some points $x_t \in \R^{n+1}$ and it holds 
\[
\sup_{t \in (0,T)} \sup_{x \in \pa E_t} H_{E_t}^s(x) \leq C,
\] 
where the constant is independent of the time of existence $T$.  We remark that by these results, we may parametrize the flow via the height function over the unit sphere $\Si^n$, which means that for every $t \in [0,T)$ there is a function $h(\cdot, t): \Si^n \to \R$ such that 
\begin{equation}
\label{eq:height}
\pa E_t = \{ h(x,t)x +x_t : x \in \Si^n\} 
\end{equation}
for some $x_t \in \R^{n+1}$. Moreover there are $c>0$ and $C \geq 1$ such that  $c \leq h(x,t)\leq C$ for every $x \in \Si^n$. We show later in Proposition \ref{prop.curvaturebound}, that the above estimate on the fractional mean curvature implies that the functions $h(\cdot, t)$ are uniformly $C^{1+s}$-regular.

Recall that our aim is to improve the $C^{1+s}$-regularity into $C^{2+\alpha}$-regularity. In other words, we show that the height functions  $h(\cdot, t)$ defined in \eqref{eq:height}  are in fact, uniformly $C^{2+s+\alpha}$ regular for small $\alpha>0$. 


\subsection{H\"older norms and Schauder estimate}
The notation $\nabla^k F$ stands for the $k$:th order differential of a vector field $F : \R^{n} \to \R^m$. For  matrix $A \in \R^{m \times k}$ we denote by $|A|$ its Frobenius norm. We define the usual H\"older norms of $F : \R^n \to \R$ by
\[
\|F\|_{C^0(\R^n)} : = \sup_{x \in \R^n} |F(x)| \quad \text{and} \quad  \|F\|_{C^\beta(\R^n)} : = \sup_{x\neq y  \in \R^n} \frac{|F(y) -F(x)|}{|y-x|^\beta} + \|F\|_{C^0(\R^n)} 
\] 
for $\beta \in (0,1]$. We extend this to $C^l$-norm, for $l = k + \beta$, with $k \in \N$ and $\beta \in [0,1)$ as
\[
 \|F\|_{C^l(\R^n)} := \sum_{j=0}^k \|\nabla^j F\|_{C^0(\R^n)} + \|\nabla^k F\|_{C^\beta(\R^n)} .
\] 
We then have the standard interpolation inequality \cite[Section 2.7]{Tri} i.e., assume $l_1$ and $l_2$ are positive numbers, $\theta \in (0,1)$ and denote  
\[
l = \theta l_1 + (1-\theta) l_2.
\] 
Then there is a constant $C \geq 1$ such that for every smooth and bounded function $u : \R^n \to \R$ it holds
\begin{equation}
\label{def:interpolation-eucl}
\|u\|_{C^l(\R^n)} \leq C\|u\|_{C^{l_1}(\R^n)}^\theta\|u\|_{C^{l_2}(\R^n)}^{1-\theta}.
\end{equation}

We will mostly  deal with functions defined on the unit sphere $\Si^n$, which we interpret as the boundary of the unit ball $\pa B_1 \subset \R^{n+1}$. Since $\Si^n$ is embedded in $\R^{n+1}$, we may extend every continuously differentiable vector field $F : \Si^n \to \R^k$ to $\R^{n+1}$ and denote the extension still by $F$. We define the tangential differential of $F$ at $x \in \Si^n$ as 
\[
\nabla_\tau F(x) :=  \nabla F(x) - (\nabla F (x)  x ) \otimes  x
\] 
and note that it does not depend on the chosen extension. For real valued functions $ u: \Si^n \to \R$ we define  the tangential gradient $\nabla_\tau u$ as the transpose of the tangential differential, and   define the tangential Hessian of $u : \Si^n \to \R$ as  $\nabla_\tau^2 u(x) = \nabla_\tau (\nabla_\tau u)(x)$. If $e_i$ is a basis vector for $\R^{n+1}$ we denote 
\[
\nabla_i u = \la \nabla_\tau u , e_i \ra
\]
for the tangential partial derivative and $\nabla_j \nabla_i u = \la \nabla_\tau^2 u\,  e_j, e_i \ra$. We may then  define the third order tangential derivatives as $\nabla_\tau \nabla_j \nabla_i u$  and iteratively  for every $k  = 3,4, \dots$.  We note that such a definition of $\nabla_\tau^k u$ does not agree with the standard definition of  a covariant derivative. 

We define the H\"older norms on $\Si^n$ as in the Euclidian setting. That is for $u : \Si^n \to \R$ we define  
\[
\|u\|_{C^0(\Si^n)} : = \sup_{x \in \Si^n} |u(x)| \quad \text{and} \quad  \|u\|_{C^\beta(\Si^n)} : = \sup_{x\neq y  \in \Si^n} \frac{|u(y) -u(x)|}{|y-x|^\beta} + \|u\|_{C^0(\Si^n)} 
\] 
for $\beta \in (0,1]$ and for  $C^l$-norm, with $l = k + \beta$, $k \in \N$ and $\beta \in [0,1)$,  as
\[
 \|u\|_{C^l(\Si^n)} := \sum_{j=0}^k \|\nabla_\tau^j u\|_{C^0(\Si^n)} + \|\nabla_\tau^k u\|_{C^\beta(\Si^n)} .
\] 
The interpolation inequality \eqref{def:interpolation-eucl} extends immediately to $\Si^n$. 
\begin{proposition}
\label{label:interpolation}
Assume $l_1$ and $l_2$ are positive numbers, $\theta \in (0,1)$ and denote  
\[
l= \theta l_1 + (1-\theta) l_2.
\] 
Then there is a constant $C\geq 1$ such that for every smooth function $u : \Si^n \to \R$ it holds
\[
\|u\|_{C^l(\Si^n)} \leq C\|u\|_{C^{l_1}(\Si^n)}^\theta\|u\|_{C^{l_2}(\Si^n)}^{1-\theta}.
\] 
\end{proposition}

Let us return to  functions defined  in $\R^n$. We recall the definition of the fractional Laplacian of order $\gamma \in (0,2)$ 
\[
\Delta^\gamma u(x) = \int_{\R^n} \frac{u(y+x)- u(x)}{|y|^{n+\gamma}}\, dy, 
\]
which has to be interpreted in the principal valued sense. We  need  parabolic Schauder estimates for equations with general linear fractional-type operator of order $\gamma = 1+s$.  To this aim we say that  $A(\cdot): \R^{n} \to \R^{n}\times \R^{n}$ is \emph{symmetric and elliptic} matrix  field, if there are constants $\lambda>0$ and $\Lambda \geq 1$ such that  for every $x \in \R^{n}$ the matrix $A(x)$ is symmetric and satisfies 
\begin{equation} \label{def:matrix-elliptic}
\lambda |\xi|^2 \leq \la A(x) \xi, \xi \ra \leq \Lambda |\xi|^2 \qquad \text{for all } \, x, \xi \in \R^{n}.  
\end{equation}
For a symmetric and elliptic matrix  field $A(\cdot)$ we define linear operator as
\begin{equation} \label{def:linear-operator-1}
L_A[u](x):= \int_{\R^n}\frac{u(y+x)- u(x)}{\la A(x) y,y\ra^{\frac{n+1+s}{2}}}\, dy. 
\end{equation}
Our aim is to give a simple proof of Schauder estimate
for the Cauchy problem
\beq
\begin{cases} \label{eq:parabolic}
\pa_t u(x,t) = L_{A}[u] +f(x,t)
\\
u(x,0)=0,
\end{cases}
\eeq
where $A(\cdot, t)$ is symmetric, elliptic and uniformly H\"older continuous for every $t \in [0,T)$.

The following Schauder estimate can be found  for instance in  \cite{MP}. For the sake of completeness and for the reader's convenience, we provide a proof based  purely on PDE methods  in the Appendix.     
\begin{theorem}
\label{teo:parabolic}
Assume  $\alpha >0$ is such that  $\alpha < \min\{s,1-s\}$ and fix $T>0$. Assume that for every $t \in [0,T)$, $A(\cdot,t)$ is symmetric and elliptic, and assume that  $\sup_{t <T}\|A(\cdot,t)\|_{C^{\alpha}(\R^n)}\leq C_0$. Let  $f: \R^n \times [0,T) \to \R$ be such that  $\sup_{t <T}\|f(\cdot,t)\|_{C^{\alpha}(\R^n)}< \infty$ and assume  $u$ is  the solution of \eqref{eq:parabolic} such that $\lim_{ |x| \to \infty}u(x,t) \leq C$. Then there exists a constant $C_T$, depending  on $\alpha, \lambda, \Lambda, C_0, \gamma, T$ and on the dimension, such that 
\[
\sup_{t<T} \|u(\cdot,t)\|_{C^{1+s+\alpha}(\R^n)} \leq C_T \sup_{t<T}\|f(\cdot,t)\|_{C^{\alpha}(\R^n)}.
\]
\end{theorem}

We will prove later, in Theorem \ref{thm2:schauder}, that the analogous result holds on the sphere.

\subsection{Preliminary results}

As we already mentioned, we know that the evolving sets $E_t$  remain  convex and satisfy the a priori estimates from \cite{CSV2} stated at the beginning of this section. Next we show that these a priori estimates  imply that the sets $E_t$ are uniformly $C^{1+s}$-regular. 

\begin{proposition}
\label{prop.curvaturebound}
Assume that $E \subset \R^{n+1}$ is a convex set such that $B_r \subset E \subset B_R$ and 
\[
\sup_{x \in \pa E} H_{E}^s(x) \leq C.
\]
Then  $E$ is uniformly $C^{1+s}$-regular and the boundary can be written as 
\[
\pa E =  \{ h(x)x : x \in \Si^n\}
\] 
with $\|h\|_{C^{1+s}(\Si^n)} \leq C$. 
\end{proposition}

\begin{proof}
The assumptions on $E$ imply that there is $r_0>0$ such that for every $x \in \pa E$ we may write the boundary locally as a graph of a bounded convex function, i.e.,  $\pa E \cap B_{r_0}(x)$ is contained on graph of a convex function which is uniformly bounded. Our goal is to prove 
\begin{equation}
\label{eq:curvature1}
\sup_{y\neq x \in \pa E} \frac{\la x-y, \nu_E(x)\ra }{|y-x|^{1+s}}\leq C. 
\end{equation}
The estimate \eqref{eq:curvature1} will then  imply the claim by the following argument. Indeed, let us fix a point $x_0 \in \pa E$ and without loss of generality we may assume $x_0 = 0$ and $\nu_E(x_0) = -e_{n+1}$. As we mentioned above, we may write 
\[
\pa E \cap B_{r_0} \subset  \{ (x',u(x')) \in \R^{n+1} : x' \in B_{r_0} \subset  \R^n \}
\] 
for a bounded convex function $u : B_{r_0}\subset  \R^{n} \to \R$. Let us choose $x',y' \in B_{r_0/2} \subset \R^n$. Note that since $u$ is convex and bounded in $B_{r_0}$, it is uniformly Lipschitz continuous in $B_{r_0/2}$.  Then \eqref{eq:curvature1}  implies
\[
\begin{split}
\langle \big(y'-x'&, u(y')-u(x')\big) , (-\nabla u(x'), 1) \rangle   \\
&\leq C\sqrt{1+ |\nabla u(x')|^2}(|y'-x'|^2 + (u(y')-u(x'))^2)^{\frac{1+s}{2}} \leq C |y'-x'|^{1+s}.
\end{split}
\] 
In other words
\[
\langle  u(y')-u(x') -\nabla u(x') , y'-x' \rangle  \leq C |y'-x'|^{1+s} \quad \text{for every } \, x',y' \in  B_{r_0/2}.
\]
Since  $u$  is convex,  then $u(y')-u(x') \geq \langle  \nabla u(x') , y'-x' \rangle$.
  The above then  implies that $\|u\|_{C^{1+s}(B_{r_0/4})} \leq C$ (see e.g. {\bf(H4)} in \cite[Appendix A]{FR}). Hence, we need to prove \eqref{eq:curvature1}.

We argue by contradiction and assume that there are sets $E_k$ which satisfy the assumptions and points $y_k \neq x_k \in \pa E_k$ such that  $|y_k - x_k| = r_k \leq 2^{-4n} r_0$, with $r_k \to 0$, and  
\begin{equation}
\label{eq:curvature2}
\frac{\la x_k-y_k, \nu_{E_k}(x_k) \ra}{|y_k-x_k|^{1+s}}\geq k. 
\end{equation}
Again we may assume that $x_k = 0$, $\nu_{E_k}(0) = -e_{n+1}$.  Since $E_k$ is convex then  
\[
  E_k \subset \{ y \in \R^{n+1} : \la y-y_k, \nu_{E_k}(y_k)\ra \leq 0 \} . 
\]
Then the following wedge-type region 
\[
U_k:= \{  y \in B_{2r_k} :  y_{n+1} > 0 , \,\,   \la y-y_k, \nu_{E_k}(y_k)\ra > 0  \}
\]
is contained in the complement of $E_k$, i.e., $U_k \cap E_k = \emptyset$. Note that \eqref{eq:curvature2} implies that 
\[
\la y_k ,  e_{n+1} \ra \geq\, k r_k^{1+s}.
\]
Therefore we deduce  that  there is a uniform constant $c>0$ such that 
\begin{equation}
\label{eq:curvature3}
|U_k| \geq  c \, k  \, r_k^{n+1+s}.
\end{equation}

Recall the definition of the fractional mean curvature
\[
H_{E_k}^s(0)  = \int_{\R^{n+1}\setminus E_k} \frac{dy}{|y|^{n+1+s}} - \int_{E_k} \frac{dy}{|y|^{n+1+s}} \leq C. 
\]  
By convexity and by $\nu_{E_k}(0) = - e_{n+1}$ it holds $ E_k \subset \{ y :  y_{n+1} \geq 0\}$. Therefore, denoting by $\tilde E_k$ the reflection of $E_k$ with respect to the hyperplane $\{x_{n+1}=0\}$, we obtain that $\tilde E_k\subset E_k^c$ and $U_k \subset (E_k\cup \tilde E_k)^c$ which implies
\[
\int_{U_k}\frac{dy}{|y|^{n+1+s}}\, dy
\leq
\int_{\R^{n+1}\setminus (E_k\cup \tilde E_k)  } \frac{dy}{|y|^{n+1+s}}
= \int_{\R^{n+1}\setminus E_k} \frac{dy}{|y|^{n+1+s}} - \int_{E_k} \frac{dy}{|y|^{n+1+s}} \leq C.
\]
By the definition of the set $U_k$ it holds $U_k \subset \{ y \in \R^{n+1}:  y_{n+1} > 0\} \setminus E_k $.  Moreover, the definition of $U_k$ implies that $|y| \leq 2r_k$ for all $y \in U_k$. Therefore by \eqref{eq:curvature3} and by the above inequality we have 
\[
 C \geq \int_{U_k} \frac{dy}{|y|^{n+1+s}} \geq c \frac{|U_k|}{r_k^{n+1+s}} \geq c \, k ,
\]
where the constant $c>0$ does not depend on $k$. Hence, we obtain a contradiction when  $k \to \infty$.
\end{proof}

\definecolor{qqqqff}{rgb}{0,0,1}
\definecolor{ffvvqq}{rgb}{1,0.3333333333333333,0}
\definecolor{uuuuuu}{rgb}{0.26666666666666666,0.26666666666666666,0.26666666666666666}
\definecolor{qqwuqq}{rgb}{0,0.39215686274509803,0}
\begin{figure}
    \label{fig:enter-label}
\begin{tikzpicture}[line cap=round,line join=round,>=triangle 45,x=40cm,y=40cm]
\clip(-0.19344947284782735,-0.0509192275431028) rectangle (0.03502429833611402,0.07997650169592732);
\draw [line width=1pt] (-0.09896184334002958,0.031131615956817182)-- (-0.09896184334002958,0);
\draw[line width=1pt,color=qqwuqq,fill=qqwuqq,fill opacity=0.25]
(-0.19561305022646316,0.07997650169592732)--(-0.19561305022646316,0.07997650169592732)--(-0.1889598314723272,0.07997650169592732)--(-0.18924739990775813,0.07997650169592732)--(-0.190156778524716,0.07997650169592732)--(-0.19106615714167385,0.07997650169592732)--(-0.19197553575863172,0.07997650169592732)--(-0.19288491437558958,0.07997650169592732)--(-0.19379429299254744,0.07997650169592732)--(-0.1947036716095053,0.07997650169592732)--(-0.19561305022646316,0.07997650169592732);
\draw[line width=1pt,color=qqwuqq,fill=qqwuqq,fill opacity=0.25]
(-0.1889598314723272,0.07997650169592732)--(-0.18833802129080027,0.07997650169592732)--(-0.1874286426738424,0.07997650169592732)--(-0.18651926405688454,0.07997650169592732)--(-0.18560988543992668,0.07996531530817765)--(-0.1847005068229688,0.0793783616142047)--(-0.18379112820601096,0.07879285108977155)--(-0.18288174958905307,0.07820878730081773)--(-0.1819723709720952,0.07762617383984767)--(-0.18106299235513734,0.07704501432626193)--(-0.18015361373817948,0.07646531240669456)--(-0.17924423512122162,0.07588707175535644)--(-0.17833485650426376,0.0753102960743845)--(-0.1774254778873059,0.07473498909419733)--(-0.17651609927034803,0.0741611545738571)--(-0.17560672065339017,0.07358879630143803)--(-0.1746973420364323,0.07301791809440145)--(-0.17378796341947444,0.07244852379997785)--(-0.17287858480251658,0.07188061729555568)--(-0.17196920618555872,0.07131420248907755)--(-0.17105982756860086,0.07074928331944354)--(-0.170150448951643,0.07018586375692215)--(-0.16924107033468513,0.06962394780356887)--(-0.16833169171772727,0.0690635394936526)--(-0.1674223131007694,0.06850464289409015)--(-0.16651293448381155,0.067947262104889)--(-0.16560355586685369,0.0673914012595984)--(-0.16469417724989582,0.06683706452576935)--(-0.16378479863293796,0.06628425610542324)--(-0.1628754200159801,0.06573298023552968)--(-0.1619660413990222,0.06518324118849364)--(-0.16105666278206437,0.06463504327265225)--(-0.16014728416510648,0.06408839083278108)--(-0.15923790554814865,0.06354328825061101)--(-0.15832852693119076,0.0629997399453548)--(-0.1574191483142329,0.06245775037424477)--(-0.15650976969727504,0.061917324033080964)--(-0.15560039108031717,0.06137846545679061)--(-0.1546910124633593,0.060841179219999014)--(-0.15378163384640145,0.06030546993761212)--(-0.1528722552294436,0.05977134226541112)--(-0.15196287661248573,0.059238800900659534)--(-0.15105349799552786,0.05870785058272274)--(-0.15014411937857,0.05817849609370088)--(-0.14923474076161214,0.05765074225907481)--(-0.14832536214465428,0.05712459394836611)--(-0.14741598352769641,0.05660005607581104)--(-0.14650660491073855,0.05607713360104904)--(-0.1455972262937807,0.055555831529826315)--(-0.14468784767682283,0.05503615491471459)--(-0.14377846905986497,0.05451810885584578)--(-0.1428690904429071,0.05400169850166281)--(-0.14195971182594924,0.05348692904968724)--(-0.14105033320899138,0.05297380574730396)--(-0.14014095459203352,0.052462333892563634)--(-0.13923157597507563,0.05195251883500332)--(-0.1383221973581178,0.05144436597648587)--(-0.1374128187411599,0.05093788077205845)--(-0.13650344012420207,0.050433068730831185)--(-0.13559406150724418,0.04992993541687592)--(-0.13468468289028634,0.049428486450146435)--(-0.13377530427332845,0.04892872750741996)--(-0.13286592565637062,0.04843066432326143)--(-0.13195654703941273,0.04793430269101044)--(-0.13104716842245487,0.04743964846379237)--(-0.130137789805497,0.04694670755555361)--(-0.12922841118853914,0.04645548594212237)--(-0.12831903257158128,0.04596598966229544)--(-0.12740965395462342,0.04547822481895184)--(-0.12650027533766556,0.04499219758019435)--(-0.1255908967207077,0.04450791418051958)--(-0.12468151810374983,0.044025380922017764)--(-0.12377213948679197,0.043544604175603005)--(-0.1228627608698341,0.043065590382275205)--(-0.12195338225287623,0.04258834605441459)--(-0.12104400363591837,0.042112877777109925)--(-0.12013462501896051,0.041639192209521635)--(-0.11922524640200265,0.041167296086280954)--(-0.11831586778504478,0.040697196218926396)--(-0.11740648916808692,0.04022889949737868)--(-0.11649711055112906,0.03976241289145574)--(-0.1155877319341712,0.039297743452428896)--(-0.11467835331721334,0.038834898314621856)--(-0.11376897470025547,0.03837388469705404)--(-0.11285959608329761,0.03791470990512974)--(-0.11195021746633975,0.03745738133237483)--(-0.11104083884938187,0.03700190646222277)--(-0.11013146023242402,0.03654829286985173)--(-0.10922208161546616,0.03609654822407462)--(-0.10831270299850829,0.035646680289284166)--(-0.10740332438155042,0.035198696927455024)--(-0.10649394576459256,0.03475260610020501)--(-0.1055845671476347,0.034308415870917916)--(-0.10467518853067684,0.03386613440693014)--(-0.10376580991371898,0.033425769981783676)--(-0.10285643129676111,0.032987330977548024)--(-0.10194705267980325,0.03255082588721387)--(-0.10103767406284539,0.03211626331716126)--(-0.10012829544588753,0.031683651989705415)--(-0.09921891682892967,0.03125300074572325)--(-0.0983095382119718,0.030824318547364022)--(-0.09740015959501394,0.030397614480847438)--(-0.09649078097805606,0.029972897759353113)--(-0.0955814023610982,0.02955017772600507)--(-0.09467202374414034,0.02912946385695533)--(-0.09376264512718248,0.02871076576457105)--(-0.09285326651022462,0.028294093200729543)--(-0.09194388789326675,0.027879456060225936)--(-0.09103450927630889,0.02746686438429858)--(-0.09012513065935103,0.0270563283642774)--(-0.08921575204239317,0.026647858345360747)--(-0.0883063734254353,0.026241464830526702)--(-0.08739699480847744,0.025837158484584996)--(-0.08648761619151958,0.02543495013837611)--(-0.08557823757456172,0.025034850793124504)--(-0.08466885895760384,0.0246368716249534)--(-0.08375948034064598,0.024241023989568785)--(-0.08285010172368812,0.023847319427120966)--(-0.08194072310673026,0.023455769667252383)--(-0.0810313444897724,0.023066386634341015)--(-0.08012196587281453,0.022679182452949122)--(-0.07921258725585667,0.022294169453487903)--(-0.07830320863889881,0.02191136017810918)--(-0.07739383002194095,0.021530767386835736)--(-0.07648445140498307,0.02115240406394327)--(-0.07557507278802521,0.020776283424607016)--(-0.07466569417106735,0.020402418921827526)--(-0.07375631555410948,0.020030824253650926)--(-0.07284693693715162,0.019661513370699622)--(-0.07193755832019376,0.01929450048403132)--(-0.0710281797032359,0.01892980007334448)--(-0.07011880108627803,0.018567426895550472)--(-0.06920942246932017,0.018207395993733663)--(-0.06830004385236231,0.017849722706522183)--(-0.06739066523540445,0.017494422677894046)--(-0.06648128661844659,0.017141511867444982)--(-0.06557190800148872,0.016791006561146155)--(-0.06466252938453085,0.01644292338262236)--(-0.06375315076757299,0.016097279304983596)--(-0.06284377215061512,0.015754091663245255)--(-0.06193439353365726,0.015413378167375218)--(-0.0610250149166994,0.015075156916009103)--(-0.06011563629974154,0.01473944641087839)--(-0.059206257682783675,0.014406265571999716)--(-0.05829687906582581,0.014075633753677657)--(-0.05738750044886795,0.013747570761378213)--(-0.05647812183191009,0.013422096869534446)--(-0.05556874321495222,0.013099232840351888)--(-0.05465936459799436,0.01277899994368694)--(-0.053749985981036495,0.012461419978078205)--(-0.05284060736407863,0.012146515293018358)--(-0.05193122874712077,0.01183430881256214)--(-0.05102185013016291,0.011524824060375319)--(-0.05011247151320504,0.01121808518633983)--(-0.04920309289624718,0.010914116994841476)--(-0.048293714279289315,0.010612944974879775)--(-0.04738433566233145,0.01031459533215335)--(-0.04647495704537359,0.010019095023291152)--(-0.04556557842841575,0.0097264717924172)--(-0.04465619981145789,0.009436754210257763)--(-0.0437468211945,0.009149971716022758)--(-0.042837442577542136,0.008866154662319622)--(-0.04192806396058427,0.008585334363388087)--(-0.04101868534362641,0.008307543146978295)--(-0.04010930672666855,0.008032814410233865)--(-0.03919992810971069,0.007761182679986678)--(-0.03829054949275282,0.007492683677921241)--(-0.037381170875794956,0.007227354391126492)--(-0.036471792258837094,0.0069652331486216084)--(-0.03556241364187923,0.006706359704522394)--(-0.03465303502492137,0.0064507753286079115)--(-0.03374365640796351,0.006198522905156514)--(-0.03283427779100564,0.005949647041047929)--(-0.031924899174047776,0.005704194184279291)--(-0.031015520557089914,0.0054622127542219)--(-0.03010614194013205,0.0052237532851575575)--(-0.02919676332317419,0.004988868584888174)--(-0.028287384706216324,0.004757613910516673)--(-0.02737800608925846,0.00453004716386655)--(-0.0264686274723006,0.004306229109454482)--(-0.025559248855342734,0.0040862236184775035)--(-0.02464987023838487,0.0038700979429485026)--(-0.023740491621427034,0.0036579230249465136)--(-0.02283111300446917,0.0034497738469862808)--(-0.021921734387511306,0.0032457298308165455)--(-0.02101235577055342,0.0030458752936106193)--(-0.020102977153595557,0.002850299972628227)--(-0.019193598536637692,0.002659099632159548)--(-0.01828421991967983,0.002472376770129436)--(-0.017374841302721968,0.0022902414464476056)--(-0.016465462685764102,0.002112812261487398)--(-0.01555608406880624,0.0019402175216091642)--(-0.014646705451848378,0.0017725966403887912)--(-0.013737326834890514,0.0016101018406470358)--(-0.01282794821793265,0.0014529002458047466)--(-0.011918569600974788,0.0013011764831948622)--(-0.011009190984016924,0.0011551359727853365)--(-0.01009981236705906,0.001015009152548792)--(-0.009190433750101198,0.0008810570144507334)--(-0.008281055133143334,0.0007535785256865969)--(-0.007371676516185471,0.00063292085179382)--(-0.006462297899227608,0.0005194939108202858)--(-0.005552919282269745,0.00041379195315520857)--(-0.004643540665311881,0.00031642725210312146)--(-0.003734162048354018,0.00022818639179440103)--(-0.002824783431396155,0.00015013354376701765)--(-0.0019154048144383162,0.00008382833920228342)--(-0.001006026197480453,0.00003190905445874816)--(-0.00009664758052258965,9.501375446147487E-7)--(0.0008127310364352735,0.0000231696923519956)--(0.0017221096533931368,0.00007146463860171634)--(0.0026314882703510245,0.00013499017420099595)--(0.003540866887308888,0.00021069993459067138)--(0.004450245504266751,0.00029687659153153183)--(0.005359624121224614,0.000392375154324545)--(0.006269002738182478,0.0004963610439426305)--(0.007178381355140341,0.0006081907211960823)--(0.008087759972098204,0.0007273482206334088)--(0.008997138589056068,0.0008534078146960914)--(0.00990651720601393,0.000986010403514144)--(0.010815895822971794,0.001124847748422728)--(0.011725274439929658,0.0012696514845330022)--(0.01263465305688752,0.001420185196052072)--(0.013544031673845384,0.0015762385365068522)--(0.014453410290803246,0.0017376227614013622)--(0.01536278890776111,0.0019041672647406227)--(0.016272167524718974,0.0020757168466797415)--(0.017181546141676836,0.0022521295250257757)--(0.018090924758634698,0.0024332747588239577)--(0.019000303375592564,0.0026190319893068946)--(0.019909681992550426,0.002809289428815366)--(0.020819060609508288,0.0030039430459921075)--(0.021728439226466154,0.0032028957081454263)--(0.02263781784342399,0.0034060564508007927)--(0.023547196460381854,0.003613339851167068)--(0.02445657507733972,0.003824665487247127)--(0.02536595369429758,0.004039957468102272)--(0.026275332311255468,0.004259144023668591)--(0.027184710928213333,0.004482157144755255)--(0.028094089545171196,0.004708932265596249)--(0.029003468162129058,0.004939407982698158)--(0.029912846779086923,0.005173525804816171)--(0.030822225396044785,0.005411229929762254)--(0.03173160401300265,0.005652467044452557)--(0.03264098262996051,0.005897186145172341)--(0.033550361246918375,0.0061453383755036735)--(0.03445973986387624,0.006396876879744545)--(0.0353691184808341,0.006651756669966287)--(0.03627849709779196,0.006909934505119574)--(0.03718787571474983,0.007171368780820395)--(0.03718787571474983,0.07997650169592732);
\draw[line width=1pt,color=ffvvqq,fill=ffvvqq,fill opacity=0.25]
(-0.00009664758052258965,-9.501375446147487E-7)--(-0.001006026197480453,-0.00003190905445874816)--(-0.0019154048144383162,-0.00008382833920228342)--(-0.002824783431396155,-0.00015013354376701765)--(-0.003734162048354018,-0.00022818639179440103)--(-0.004643540665311881,-0.00031642725210312146)--(-0.005552919282269745,-0.00041379195315520857)--(-0.006462297899227608,-0.0005194939108202858)--(-0.007371676516185471,-0.00063292085179382)--(-0.008281055133143334,-0.0007535785256865969)--(-0.009190433750101198,-0.0008810570144507334)--(-0.01009981236705906,-0.001015009152548792)--(-0.011009190984016924,-0.0011551359727853365)--(-0.011918569600974788,-0.0013011764831948622)--(-0.01282794821793265,-0.0014529002458047466)--(-0.013737326834890514,-0.0016101018406470358)--(-0.014646705451848378,-0.0017725966403887912)--(-0.01555608406880624,-0.0019402175216091642)--(-0.016465462685764102,-0.002112812261487398)--(-0.017374841302721968,-0.0022902414464476056)--(-0.01828421991967983,-0.002472376770129436)--(-0.019193598536637692,-0.002659099632159548)--(-0.020102977153595557,-0.002850299972628227)--(-0.02101235577055342,-0.0030458752936106193)--(-0.021921734387511306,-0.0032457298308165455)--(-0.02283111300446917,-0.0034497738469862808)--(-0.023740491621427034,-0.0036579230249465136)--(-0.02464987023838487,-0.0038700979429485026)--(-0.025559248855342734,-0.0040862236184775035)--(-0.0264686274723006,-0.004306229109454482)--(-0.02737800608925846,-0.00453004716386655)--(-0.028287384706216324,-0.004757613910516673)--(-0.02919676332317419,-0.004988868584888174)--(-0.03010614194013205,-0.0052237532851575575)--(-0.031015520557089914,-0.0054622127542219)--(-0.031924899174047776,-0.005704194184279291)--(-0.03283427779100564,-0.005949647041047929)--(-0.03374365640796351,-0.006198522905156514)--(-0.03465303502492137,-0.0064507753286079115)--(-0.03556241364187923,-0.006706359704522394)--(-0.036471792258837094,-0.0069652331486216084)--(-0.037381170875794956,-0.007227354391126492)--(-0.03829054949275282,-0.007492683677921241)--(-0.03919992810971069,-0.007761182679986678)--(-0.04010930672666855,-0.008032814410233865)--(-0.04101868534362641,-0.008307543146978295)--(-0.04192806396058427,-0.008585334363388087)--(-0.042837442577542136,-0.008866154662319622)--(-0.0437468211945,-0.009149971716022758)--(-0.04465619981145789,-0.009436754210257763)--(-0.04556557842841575,-0.0097264717924172)--(-0.04647495704537359,-0.010019095023291152)--(-0.04738433566233145,-0.01031459533215335)--(-0.048293714279289315,-0.010612944974879775)--(-0.04920309289624718,-0.010914116994841476)--(-0.05011247151320504,-0.01121808518633983)--(-0.05102185013016291,-0.011524824060375319)--(-0.05193122874712077,-0.01183430881256214)--(-0.05284060736407863,-0.012146515293018358)--(-0.053749985981036495,-0.012461419978078205)--(-0.05465936459799436,-0.01277899994368694)--(-0.05556874321495222,-0.013099232840351888)--(-0.05647812183191009,-0.013422096869534446)--(-0.05738750044886795,-0.013747570761378213)--(-0.05829687906582581,-0.014075633753677657)--(-0.059206257682783675,-0.014406265571999716)--(-0.06011563629974154,-0.01473944641087839)--(-0.0610250149166994,-0.015075156916009103)--(-0.06193439353365726,-0.015413378167375218)--(-0.06284377215061512,-0.015754091663245255)--(-0.06375315076757299,-0.016097279304983596)--(-0.06466252938453085,-0.01644292338262236)--(-0.06557190800148872,-0.016791006561146155)--(-0.06648128661844659,-0.017141511867444982)--(-0.06739066523540445,-0.017494422677894046)--(-0.06830004385236231,-0.017849722706522183)--(-0.06920942246932017,-0.018207395993733663)--(-0.07011880108627803,-0.018567426895550472)--(-0.0710281797032359,-0.01892980007334448)--(-0.07193755832019376,-0.01929450048403132)--(-0.07284693693715162,-0.019661513370699622)--(-0.07375631555410948,-0.020030824253650926)--(-0.07466569417106735,-0.020402418921827526)--(-0.07557507278802521,-0.020776283424607016)--(-0.07648445140498307,-0.02115240406394327)--(-0.07739383002194095,-0.021530767386835736)--(-0.07830320863889881,-0.02191136017810918)--(-0.07921258725585667,-0.022294169453487903)--(-0.08012196587281453,-0.022679182452949122)--(-0.0810313444897724,-0.023066386634341015)--(-0.08194072310673026,-0.023455769667252383)--(-0.08285010172368812,-0.023847319427120966)--(-0.08375948034064598,-0.024241023989568785)--(-0.08466885895760384,-0.0246368716249534)--(-0.08557823757456172,-0.025034850793124504)--(-0.08648761619151958,-0.025434950138376107)--(-0.08739699480847744,-0.025837158484584996)--(-0.0883063734254353,-0.026241464830526702)--(-0.08921575204239317,-0.026647858345360747)--(-0.09012513065935103,-0.0270563283642774)--(-0.09103450927630889,-0.02746686438429858)--(-0.09194388789326675,-0.02787945606022593)--(-0.09285326651022462,-0.028294093200729543)--(-0.09376264512718248,-0.02871076576457105)--(-0.09467202374414034,-0.02912946385695533)--(-0.0955814023610982,-0.029550177726005076)--(-0.09649078097805606,-0.029972897759353107)--(-0.09740015959501394,-0.030397614480847445)--(-0.0983095382119718,-0.03082431854736403)--(-0.09921891682892967,-0.031253000745723264)--(-0.10012829544588753,-0.03168365198970541)--(-0.10103767406284539,-0.03211626331716126)--(-0.10194705267980325,-0.03255082588721388)--(-0.10285643129676111,-0.03298733097754803)--(-0.10376580991371898,-0.03342576998178367)--(-0.10467518853067684,-0.03386613440693014)--(-0.1055845671476347,-0.034308415870917916)--(-0.10649394576459256,-0.034752606100205)--(-0.10740332438155042,-0.035198696927455024)--(-0.10831270299850829,-0.035646680289284166)--(-0.10922208161546616,-0.036096548224074605)--(-0.11013146023242402,-0.03654829286985172)--(-0.11104083884938187,-0.03700190646222277)--(-0.11195021746633975,-0.03745738133237482)--(-0.11285959608329761,-0.037914709905129754)--(-0.11376897470025547,-0.038373884697054036)--(-0.11467835331721334,-0.03883489831462186)--(-0.1155877319341712,-0.03929774345242891)--(-0.11649711055112906,-0.039762412891455746)--(-0.11740648916808692,-0.04022889949737868)--(-0.11831586778504478,-0.040697196218926396)--(-0.11922524640200265,-0.04116729608628096)--(-0.12013462501896051,-0.041639192209521635)--(-0.12104400363591837,-0.04211287777710992)--(-0.12195338225287623,-0.04258834605441459)--(-0.1228627608698341,-0.0430655903822752)--(-0.12377213948679197,-0.04354460417560299)--(-0.12468151810374983,-0.04402538092201775)--(-0.1255908967207077,-0.04450791418051958)--(-0.12650027533766556,-0.04499219758019434)--(-0.12740965395462342,-0.04547822481895184)--(-0.12831903257158128,-0.04596598966229544)--(-0.12922841118853914,-0.046455485942122385)--(-0.130137789805497,-0.046946707555553614)--(-0.13104716842245487,-0.047439648463792373)--(-0.13195654703941273,-0.04793430269101044)--(-0.13286592565637062,-0.04843066432326144)--(-0.13377530427332845,-0.04892872750741997)--(-0.13468468289028634,-0.04942848645014643)--(-0.13559406150724418,-0.049929935416875924)--(-0.13650344012420207,-0.05043306873083117)--(-0.1374128187411599,-0.0509192275431028)--(-0.1383221973581178,-0.0509192275431028)--(-0.13923157597507563,-0.0509192275431028)--(-0.14014095459203352,-0.0509192275431028)--(-0.14105033320899138,-0.0509192275431028)--(-0.14124348554432023,-0.0509192275431028)--(0.03718787571474983,-0.0509192275431028)--(0.03718787571474983,-0.007171368780820395)--(0.03627849709779196,-0.006909934505119574)--(0.0353691184808341,-0.006651756669966287)--(0.03445973986387624,-0.006396876879744545)--(0.033550361246918375,-0.0061453383755036735)--(0.03264098262996051,-0.005897186145172341)--(0.03173160401300265,-0.005652467044452557)--(0.030822225396044785,-0.005411229929762254)--(0.029912846779086923,-0.005173525804816171)--(0.029003468162129058,-0.004939407982698158)--(0.028094089545171196,-0.004708932265596249)--(0.027184710928213333,-0.004482157144755255)--(0.026275332311255468,-0.004259144023668591)--(0.02536595369429758,-0.004039957468102272)--(0.02445657507733972,-0.003824665487247127)--(0.023547196460381854,-0.003613339851167068)--(0.02263781784342399,-0.0034060564508007927)--(0.021728439226466154,-0.0032028957081454263)--(0.020819060609508288,-0.0030039430459921075)--(0.019909681992550426,-0.002809289428815366)--(0.019000303375592564,-0.0026190319893068946)--(0.018090924758634698,-0.0024332747588239577)--(0.017181546141676836,-0.0022521295250257757)--(0.016272167524718974,-0.0020757168466797415)--(0.01536278890776111,-0.0019041672647406227)--(0.014453410290803246,-0.0017376227614013622)--(0.013544031673845384,-0.0015762385365068522)--(0.01263465305688752,-0.001420185196052072)--(0.011725274439929658,-0.0012696514845330022)--(0.010815895822971794,-0.001124847748422728)--(0.00990651720601393,-0.000986010403514144)--(0.008997138589056068,-0.0008534078146960914)--(0.008087759972098204,-0.0007273482206334088)--(0.007178381355140341,-0.0006081907211960823)--(0.006269002738182478,-0.0004963610439426305)--(0.005359624121224614,-0.000392375154324545)--(0.004450245504266751,-0.00029687659153153183)--(0.003540866887308888,-0.00021069993459067138)--(0.0026314882703510245,-0.00013499017420099595)--(0.0017221096533931368,-0.00007146463860171634)--(0.0008127310364352735,-0.0000231696923519956)--(-0.00009664758052258965,-9.501375446147487E-7);
\draw[line width=1pt,color=ffvvqq,fill=ffvvqq,fill opacity=0.25]
(-0.19561305022646316,-0.0509192275431028)--(-0.19561305022646316,-0.0509192275431028)--(-0.1947036716095053,-0.0509192275431028)--(-0.19379429299254744,-0.0509192275431028)--(-0.19288491437558958,-0.0509192275431028)--(-0.19197553575863172,-0.0509192275431028)--(-0.19106615714167385,-0.0509192275431028)--(-0.190156778524716,-0.0509192275431028)--(-0.18924739990775813,-0.0509192275431028)--(-0.18833802129080027,-0.0509192275431028)--(-0.1874286426738424,-0.0509192275431028)--(-0.18651926405688454,-0.0509192275431028)--(-0.18560988543992668,-0.0509192275431028)--(-0.1847005068229688,-0.0509192275431028)--(-0.18379112820601096,-0.0509192275431028)--(-0.18288174958905307,-0.0509192275431028)--(-0.1819723709720952,-0.0509192275431028)--(-0.18106299235513734,-0.0509192275431028)--(-0.18015361373817948,-0.0509192275431028)--(-0.17924423512122162,-0.0509192275431028)--(-0.17833485650426376,-0.0509192275431028)--(-0.1774254778873059,-0.0509192275431028)--(-0.17651609927034803,-0.0509192275431028)--(-0.17560672065339017,-0.0509192275431028)--(-0.1746973420364323,-0.0509192275431028)--(-0.17378796341947444,-0.0509192275431028)--(-0.17287858480251658,-0.0509192275431028)--(-0.17196920618555872,-0.0509192275431028)--(-0.17105982756860086,-0.0509192275431028)--(-0.170150448951643,-0.0509192275431028)--(-0.16924107033468513,-0.0509192275431028)--(-0.16833169171772727,-0.0509192275431028)--(-0.1674223131007694,-0.0509192275431028)--(-0.16651293448381155,-0.0509192275431028)--(-0.16560355586685369,-0.0509192275431028)--(-0.16469417724989582,-0.0509192275431028)--(-0.16378479863293796,-0.0509192275431028)--(-0.1628754200159801,-0.0509192275431028)--(-0.1619660413990222,-0.0509192275431028)--(-0.16105666278206437,-0.0509192275431028)--(-0.16014728416510648,-0.0509192275431028)--(-0.15923790554814865,-0.0509192275431028)--(-0.15832852693119076,-0.0509192275431028)--(-0.1574191483142329,-0.0509192275431028)--(-0.15650976969727504,-0.0509192275431028)--(-0.15560039108031717,-0.0509192275431028)--(-0.1546910124633593,-0.0509192275431028)--(-0.15378163384640145,-0.0509192275431028)--(-0.1528722552294436,-0.0509192275431028)--(-0.15196287661248573,-0.0509192275431028)--(-0.15105349799552786,-0.0509192275431028)--(-0.15014411937857,-0.0509192275431028)--(-0.14923474076161214,-0.0509192275431028)--(-0.14832536214465428,-0.0509192275431028)--(-0.14741598352769641,-0.0509192275431028)--(-0.14650660491073855,-0.0509192275431028)--(-0.1455972262937807,-0.0509192275431028)--(-0.14468784767682283,-0.0509192275431028)--(-0.14377846905986497,-0.0509192275431028)--(-0.1428690904429071,-0.0509192275431028)--(-0.14195971182594924,-0.0509192275431028)--(-0.14124348554432023,-0.0509192275431028);
\draw [color=ffvvqq](0.0034360686080312157,-0.0028880673925826606) node[anchor=north west] {$\mathbf{\tilde{E}_k}$};
\draw (-0.13027301339166175,-0.0005081450427821131) node[anchor=north west] {$ \{x_{n+1}=0\}$};
\draw (-0.1233681015082123,0.03005611651324448) node[anchor=north west] {$\mathbf{\mathbf{y_k}}$};
\draw [line width=1pt,domain=-0.19344947284782735:0.03502429833611402] plot(\x,{(-0-0*\x)/0.06229517667336292});
\draw [color=qqqqff](-0.1907057288673889,0.02653642711404229) node[anchor=north west] {$\mathbf{\mathbf{U_k}}$};
\draw (-0.09522305985776165,0.01939666006464065) node[anchor=north west] {$k r_k^{1+s}$};
\draw [line width=1pt] (-0.09896184334002958,0.031131615956817182)-- (-0.03666666666666665,0);
\draw [color=qqwuqq](-0.01863242065405403,0.03886875201755422) node[anchor=north west] {$\mathbf{E_k}$};
\draw [shift={(-0.036666666666664995,0)},line width=1pt,color=qqqqff,fill=qqqqff,fill opacity=0.1]  (0,0) --  plot[domain=2.6799520791315077:3.1415926535897865,variable=\t]({1*0.13319599471855922*cos(\t r)+0*0.13319599471855922*sin(\t r)},{0*0.13319599471855922*cos(\t r)+1*0.13319599471855922*sin(\t r)}) -- cycle ;
\draw [line width=1pt] (-0.15611666543693123,0.05893106823102532)-- (-0.002405590314285467,0);
\draw (-0.07293821285781282,0.042330457253627746) node[anchor=north west] {$2r_k$};
\begin{scriptsize}
\end{scriptsize}
\end{tikzpicture}
\caption{A picture of the argument used in the proof of Proposition \ref{prop.curvaturebound}. In green we have the set $E_k$, in orange $\tilde E_k$, i.e. the reflection of $E_k$ over the plane  $\{x_{n+1}=0\}$ and in blue the set $U_k$}
\end{figure}
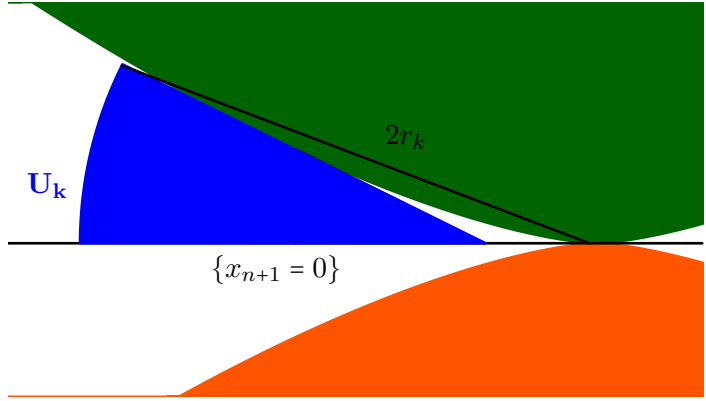

At the end of the section we prove continuity estimates for functions defined via convolution over kernels. We  define elliptic and symmetric matrix field on $\Si^n$, $A(\cdot) : \Si^n \to \R^{n+1} \times \R^{n+1}$ as in  \eqref{def:matrix-elliptic}, i.e., there are $0< \lambda < 1 <\Lambda$ such that 
\[
\lambda |\xi|^2 \leq \la A(x) \xi, \xi \ra \leq \Lambda |\xi|^2 \quad \text{for all } \, \xi \in \R^{n+1 } \, \text{and } \, x \in \Si^n. 
\] 
We consider kernels $K : \Si^n \times \Si^n \setminus \{(x,x) : x \in \Si^n\} \to \R$ of  the  type
\begin{equation}\label{def:kernel}
K_A(y,x) = \frac {1}{\la A(x) (y-x) , (y-x)\ra^{\frac{n+1+s}{2}}} ,
\end{equation}
where $A(\cdot)$ is symmetric, elliptic and $\|A\|_{C^\alpha(\Si^n)} \leq C$.  In order to simplify the notation, we denote  the induced  norm   by 
\begin{equation}\label{def:matrixnorm}
\|\xi\|_{A(x)}^2  := \la A(x) \xi , \xi \ra \qquad \text{for } \, \xi \in \R^{n+1}.
\end{equation}
We may then denote $K$ defined in \eqref{def:kernel} as $K(y,x)  = \|y-x\|_{A(x)}^{-n-1-s}$.

The following lemma is easy to prove and thus we omit it. 
\begin{lemma}
\label{lem:kernel}
Assume that $A(\cdot)$ is symmetric, elliptic and $\|A\|_{C^\alpha(\Si^n)} \leq C$ for $0 <\alpha < \min \{s,1-s \}$. Then the kernel  $K_A$ defined in \eqref{def:kernel} satisfies the following two  conditions:
\begin{itemize}
\item[(i)] $K_A$ is continuous on $ \Si^n \times \Si^n \setminus \{(x,x) : x \in \Si^n\}$ and satisfies
\[
|K_A(y,x)| \leq \frac{C}{|y-x|^{n+1+s}}.
\]
\item[(ii)] For all $x,y,z \in \Si^n$ with $0<|z-x|\leq \frac12 |y-x|$ it holds
\[
|K_A(y,z) -K_A(y,x)| \leq C \frac{|z-x|}{|y-x|^{n+2+s}}  + C \frac{|z-x|^\alpha}{|y-x|^{n+1+s}}.
\]
\end{itemize} 
\end{lemma}

We note that the above conditions agree with the assumptions in \cite[Definition 4.1]{JL} with the difference that the second condition is slightly weaker.
The following lemma is almost the same as \cite[Lemma 4.3]{JL}.
\begin{lemma}
\label{lem:old.4.3}
Assume that $A(\cdot)$ is symmetric, elliptic and $\|A\|_{C^\alpha(\Si^n)} \leq C$ for $0 <\alpha < \min \{s,1-s \}$, and  the kernel  $K_A$ is defined in \eqref{def:kernel}. Assume that the function $F \in C(\Si^n \times \Si^n)$ satisfies the following conditions. 
\begin{enumerate}
\item For all $x,y \in \Si^n$  it holds 
\[
|F(y,x)| \leq \kappa_0 |y-x|^{1+s+\alpha}.
\]
\item For all $x,y,z \in \Si^n$ with $|z-x|\leq \frac12 |y-x|$ it holds
\[
|F(y,z) -F(y,x)| \leq \kappa_0 \big( |z-x|^\alpha |y-x|^{1+s+\alpha/2} + |z-x|^{s+\alpha} |y-x| \big).
\]
\end{enumerate} 
Then the function 
\[
\psi(x) = \int_{\Si^n} F(y,x)K_A(y,x)\, d\H_y^n 
\]
is H\"older continuous and $\|\psi\|_{C^{\alpha}(\Si^n)}\leq C \kappa_0$.
\end{lemma}

\begin{proof}
 The proof is almost the same as in \cite[Lemma 4.3]{JL} but we write it for the reader's convenience since the assumptions are slightly different.   The condition (i) in Lemma \ref{lem:kernel} and the assumption (1) for $F$ immediately implies that $\|\psi\|_{C^0}\leq C \kappa_0$. To show the H\"older continuity we fix $z,x \in \Si^n$ and divide the sphere into two sets as  $S_- = \{y \in  \Si^n :  |y-x| \leq 2 |z-x| \}$ and $S_+ = \{y \in  \Si^n :  |y-x| > 2 |z-x| \}$. Note that for $y \in S_-$ it holds $|y-z| \leq |y-x| + |z-x| \leq 3|z-x|$. We use the condition (i) in Lemma \ref{lem:kernel} and the assumption (1) for $F$ to deduce
\[
\int_{S_-} |F(y,z)K_A(y,z)|\, d\H_y^n  \leq  C \kappa_0 \int_{S_-} \frac{1}{|y-z|^{n-\alpha}}\, d \H_y^n \leq C \kappa_0 \int_{0}^{3|z-x|}\rho^{\alpha-1} \, d \rho = C\kappa_0 |z-x|^\alpha.
\]
Similarly it holds $\int_{S_-} |F(y,x)K_A(y,x)|\, d\H_y^n  \leq  C\kappa_0 |z-x|^\alpha$. 

Note that for $y \in S_+$ it holds $|y-z| \geq |y-x|- |z-x| \geq \frac12 |y-x|$. We use the conditions (i) and (ii) in Lemma \ref{lem:kernel}  and the assumption (1) and (2) for $F$ to estimate
\[
\begin{split}
\int_{S_+} &|F(y,z)K_A(y,z) - F(y,x)K_A(y,x)|\, d\H_y^n \\
&\leq \int_{S_+} |F(y,z) - F(y,x)| |K_A(y,z) |\, d\H_y^n + \int_{S_-} |F(y,x)| |K_A(y,z) - K_A(y,x) |\, d\H_y^n\\
&\leq C\kappa_0 |z-x|^\alpha \int_{S_+} \frac{1}{|y-x|^{n -\alpha/2}}\, d\H_y^n + C\kappa_0 |z-x|^{s+\alpha} \int_{S_+} \frac{1}{|y-x|^{n +s}}\, d\H_y^n\\ 
 &\,\,\,\,\,\,\,\,\,\,\,\,\,\,\,\,\,\,+ C\kappa_0 |z-x| \int_{S_+} \frac{1}{|y-x|^{n +1 -\alpha}}\, d\H_y^n. 
\end{split}
\]
Since $\frac{1}{|y-x|^{n -\alpha/2}}$ is not singular, we obtain $\int_{S_+} \frac{1}{|y-x|^{n -\alpha/2}}\, d\H_y^n \leq C$. Since $|z-x|\leq |y-x|$ we estimate  the last term by 
\[
\frac{|z-x|}{|y-x|^{n +1 -\alpha}}\leq \frac{|z-x|^{s+\alpha}}{|y-x|^{n +s}} 
\]
which then implies
\[
|z-x| \int_{S_+} \frac{1}{|y-x|^{n +1 -\alpha}}\, d\H_y^n \leq |z-x|^{s+\alpha} \int_{S_+} \frac{1}{|y-x|^{n +s}}\, d\H_y^n \leq C|z-x|^{s+\alpha}   \int_{|z-x|}^2 \rho^{-s-1} \leq  C |z-x|^{\alpha}.  
\] 
Therefore we have 
\[
\int_{S_+} |F(y,z)K_A(y,z) - F(y,x)K_A(y,x)|\, d\H_y^n \leq C \kappa_0 |z-x|^\alpha
\]
and the claim follows. 
\end{proof}

\begin{remark} \label{rem:lem2.5-eucl}
The above lemma holds also in $\R^n$ for a kernel defined as in \eqref{def:kernel}, if we change the assumptions on $F : \R^n\times \R^n \to \R$ as follows. 
\begin{enumerate}
\item For all $x \in \R^n$  it holds 
\[
|F(y,x)| \leq \kappa_0 \min\{|y-x|, 1\}^{1+s+\alpha}.
\]
\item For all $x,y,z \in \R^n$ with $|z-x|\leq \frac12 |y-x|$ it holds
\[
|F(y,z) -F(y,x)| \leq \kappa_0 \big( |z-x|^\alpha \min\{|y-x|, 1\}^{1+s+\alpha/2} + |z-x|^{s+\alpha} \min\{|y-x|, 1\} \big).
\]
\end{enumerate} 

\end{remark}

\section{Equation for the height function}

As we mentioned in the introduction, instead of the original equation \eqref{eq:the-flow}, we will differentiate it with respect to space and parametrize the equation
\begin{equation}\label{eq:diff-flow}
\nabla_{X(t)} V_t = - \nabla_{X(t)} H_{E_t^s}  \qquad \text{on } \, \pa E_t,
\end{equation}
where $X(t)$ is a vector field on the moving boundary $\pa E_t$, which we will define shortly, and $\nabla_{X(t)}$ is the derivative in the direction of $X(t)$. Note that the derivative of the integral average $\bar H_{E_t}^s$ vanishes. 

We parametrize the flow using the height function similarly as in \eqref{eq:height}, with the difference  that we assume we may write locally in time  
\[
\pa E_t = \{ h(x,t) x + x_0 : x \in \Si^n\}, 
\]
for some fixed $x_0 \in \R^{n+1}$ and for $t \in [0,T)$. We denote 
\[
\eta(x,t) := h(x,t) x.  
\] 
 We choose the vector field $X(t)$ as follows. Fix a basis vector $e_i,\,  i=1, \dots, n+1$ of the ambient space $\R^{n+1}$. We obtain a tangent field on $\Si^n$ as
\[
\tau_i(x) = e_i - x_i \, x, \qquad \text{for } \, x \in \Si^n, 
\]
where $x_i = \la x,  e_i \ra$.   This induces a tangent field on the moving boundary $\pa E_t$, which we denote by $X(t)= X(\cdot, t)$,  defined as
\[
X(\eta(x,t), t) := D \eta(x,t) \tau_i(x) = h(x,t)  \tau_i(x)  +\nabla_i h(x,t) x \qquad \text{for }  \, x \in \Si^n.
\]
We may then write the derivative $\nabla_{X(t)} f$ of any given function $f \in C^1(\partial E_t)$ at the point $\eta(x,t) \in \partial E_t$ for $x \in \Si^n$ as
\[
(\nabla_{X(t)} f)(\eta(x,t)) = \nabla_i (f\circ \eta(\cdot, t)) (x) .
\]
We will assume throughout the section that the height function satisfies the uniform bound $\|h(\cdot, t)\|_{C^{1+s}}\leq C$ and $c \leq h(x,t) \leq C$  for all $t$ and $x \in \Si^n$. 

Let us first deal the LHS of \eqref{eq:diff-flow}. This is rather easy as we may write the normal as
\[
\nu_{E_t}(\eta(x,t)) = \frac{h(x,t) x - \nabla_\tau h(x,t)}{\sqrt{h^2(x,t) + |\nabla_\tau h(x,t)|^2}} 
\] 
and the normal velocity then is 
\begin{equation}\label{eq:normal}
V_t(\eta(x,t)) = \pa_t h(x,t) \, \frac{h(x,t)}{\sqrt{h^2(x,t) + |\nabla_\tau h(x,t)|^2}}. 
\end{equation}
Hence, it holds 
\begin{equation}\label{eq:normal-deri}
\begin{split}
(\nabla_{X(t)} V_t) (\eta(x,t)) &= \nabla_i \left(\pa_t h(x,t) \, \frac{h(x,t)}{\sqrt{h^2(x,t) + |\nabla_\tau h(x,t)|^2}}\right)\\
&=    \pa_t  \nabla_i h(x,t) \, \frac{h(x,t)}{\sqrt{h^2(x,t) + |\nabla_\tau h(x,t)|^2}} \\
&\,\,\,\,\,\,\,\,\,\,\,\,\,\,\,+ \pa_t h(x,t) \nabla_i \left(\frac{h(x,t)}{\sqrt{h^2(x,t) + |\nabla_\tau h(x,t)|^2}}\right) .
\end{split}
\end{equation}

Let us then focus on the RHS of \eqref{eq:diff-flow}. We recall that following the argument in  \cite{FFMMM, SaeV} we may write 
\begin{equation}\label{eq:deri-MC}
\nabla_{X(t)} H_{E_t}^s  = 2 \int_{\pa E_t}   
 \frac{\la X(t), \nu_{E_t}(y) \ra}{|y-x|^{n+1+s}} \, d\H_y^n.
\end{equation}
Indeed, by regularizing the kernel, say by $K_\delta(t) = (t^2 +\delta)^{-\frac{n+1+s}{2}}$, and  differentiating in $x_i$ direction 
(see \cite[Formula (6.17)]{FFMMM})  we have 
\[
\begin{split}
\nabla_{x_i} &\left(\int_{E_t^c} K_{\delta}(|y-x|) \, dy   - \int_{E_t} K_{\delta}(|y-x|) \, dy  \right)\\
&=  - \int_{E_t^c} K_{\delta}'(|y-x|) \frac{\langle y-x, e_i\rangle }{|y-x|} \, dy   +  \int_{E_t} K_{\delta}'(|y-x|) \frac{\langle y-x, e_i\rangle }{|y-x|} \, dy  \\
&= - \int_{E_t^c} \text{div}_y (K_{\delta}(|y-x|)e_i) \, dy   +  \int_{E_t} \text{div}_y (K_{\delta}(|y-x|)e_i) \, dy  \\
&= \int_{\pa E_t} 2K_{\delta}(|y-x|) \langle e_i, \nu_{E_t}(y)\rangle   \, d\H_y^n,
\end{split}
\]
where the last equality follows from divergence theorem. Letting $\delta \to 0$ yields \eqref{eq:deri-MC}. 

We  parametrize  \eqref{eq:deri-MC} on $\Si^n$, by recalling the formula for the  normal $\nu_{E_t}(\eta(x,t))$ above  and the Jacobian determinant of $\eta$ as
\[
J_{\eta(\cdot, t)}(x) = h(x,t)^{n-1}  \sqrt{h^2(x,t) + |\nabla_\tau h(x,t)|^2}.
\]
 For simplicity we omit the time-dependence for a moment and write $h(x) = h(x,t)$. Then we may write
\[
\nabla_{X(t)} H_{E_t}^s  =2 \int_{\Si^n} h(y)^{n-1}\frac{\la h(x)\tau_i(x) + \nabla_i h(x) x, h(y)y - \nabla_\tau h(y) \ra }{|h(y) y - h(x)x|^{n+1+s}} \, d\H_y^n 
\] 
We continue and write  the terms as 
\[
\begin{split}
 h(x)\tau_i(x) + \nabla_i h(x) x =  &h(y)\tau_i(y) + \nabla_i h(x) y  \\
&+  \big(h(x)\tau_i(x) -  h(y)\tau_i(y) \big)  + ( x-y ) \nabla_i h(x).
\end{split}
\]
This gives us the following expression
\[
\begin{split}
\nabla_{X(t)} H_{E_t}^s  = 2 \int_{\Si^n} h(y)^{n-1}\frac{\la h(y)\tau_i(y) + \nabla_i h(x) y , h(y)y - \nabla_\tau h(y) \ra }{|h(y) y - h(x)x|^{n+1+s}} \, d\H_y^n  +R_1
\end{split}
\]
where the error terms are 
\[
\begin{split}
R_1  &= 2 \int_{\Si^n}\frac{\la h(x)\tau_i(x) -  h(y)\tau_i(y)  , h(y)y - \nabla_\tau h(y) \ra }{|h(y) y - h(x)x|^{n+1+s}}   h(y)^{n-1} \, d\H_y^n \\
&+ 2 \int_{\Si^n} \frac{\la  x-y   , h(y)y - \nabla_\tau h(y) \ra }{|h(y) y - h(x)x|^{n+1+s}}  \nabla_i h(x)  h(y)^{n-1} \, d\H_y^n .
\end{split}
\]
Note that we may write the error terms as 
\[
R_1 = \sum_k a_k(x,h,\nabla_\tau h) \int_{\Si^n} \frac {F_k (y,h(y))    -F_k(x,h(x))}{|h(y)y - h(x)x|^{n+1+s}}  b_k(y,h,\nabla_\tau h)\, d\H_y^n
\]
for smooth functions $F_k, a_k, b_k$.  But since 
\[
\la h(y)\tau_i(y) + \nabla_i h(x) y , h(y)y - \nabla_\tau h(y) \ra =  \big(\nabla_i h(x)  - \nabla_i h(y) \big) h(y)
\]
we obtain 
\begin{equation} \label{eq:deriv-RHS}
-\nabla_{X(t)} H_{E_t}^s  =2 \int_{\Si^n} \frac{\nabla_i h(y) - \nabla_i h(x) }{|h(y) y - h(x)x|^{n+1+s}} h(y)^{n} \, d\H_y^n  + R_1.
\end{equation}

We  further linearize the first term on the RHS of \eqref{eq:deriv-RHS}, which is the leading order term in the equation.  We  write  using $\la x-y, y \ra = -\frac12 |x-y|^2$
\[
\begin{split}
|h(y)y -h(x)x |^2 &= |h(x)(x-y) + (h(x) -h(y))y|^2 \\
&= h^2(x)|y-x|^2 + h(x)(h(y) -h(x))|x- y|^2 + |h(y)-h(x)|^2 .
\end{split}
\]
We write the Taylor expansion of $h$ at $x$  as
\[
T_x[h](y) = h(y) -h(x) - \la \nabla_\tau h(x) ,  y-x \ra
\]
and note that for $\beta \in [0,1]$ it holds 
\beq\label{eq:taylor1}
|T_x[h](y)| \leq \|h\|_{C^{1+\beta}}|y-x|^{1+\beta}.
\eeq
Therefore we deduce that we may write  
\beq\label{eq:split-by-g}
\begin{split}
|h(y)y -h(x)x |^2 &= h^2(x)|y-x|^2 + \la  \nabla_\tau h(x),  y-x\ra^2 + g_h(y,x)\\
&= \|x-y\|_{A(x)}^2+ g_h(y,x)
\end{split}
\eeq
where  the matrix is given by tensor product $A(x) = h^2(x)I + \nabla_\tau h(x) \otimes \nabla_\tau h(x)$, the norm $\|\xi \|_{A(x)}$ is   defined in \eqref{def:matrixnorm}, and the error term is
\[
g_h(y,x) = h(x)(h(y) -h(x))|x- y|^2 + 2(h(y) -h(x)) T_x[h](y)  - |T_x[h](y)|^2.  
\]
The precise form of $g_h$ is not  relevant. Instead, it is important to notice that it satisfies the following two conditions.
 For all $x,y \in \Si^n$ and $0 < \alpha <\min\{s, 1-s \}$ it holds 
\begin{equation}\label{eq:g_h-1}
|g_h(y,x)| \leq C\|h\|_{C^{1+s+\alpha}} |y-x|^{2+s+ \alpha}
\end{equation}
 and for all $x,y,z \in \Si^n$ with $|z-x|\leq \frac12 |y-x|$ it holds
\begin{equation}\label{eq:g_h-2}
|g_h(y,z) -g_h(y,x)| \leq C\|h\|_{C^{1+s+\alpha}} |z-x|^{s+\alpha} |y-x|^2 .
\end{equation}
Indeed, the estimate \eqref{eq:g_h-1} follows by using  \eqref{eq:taylor1} for $\beta = 0$ and $\beta = s+\alpha$, and  recalling that  $\|h\|_{C^{1+s}} \leq C$ . The estimate \eqref{eq:g_h-2} follows from the fact that for $|z-x|\leq \frac12 |y-x|$ it holds
\begin{equation}\label{eq:g_h-3}
\begin{split}
\big|T_z[h](y) -T_x[h](y) \big|   &= \big| \big( h(y) - h(z)  - \la  \nabla_\tau h(z) , y-z\ra  \big)-  \big(h(y) - h(x)  - \la  \nabla_\tau h(x) , y-x\ra \big) \big| \\
&= \big| \big(h(x) - h(z)  - \la \nabla_\tau h(z) , x-z \ra\big) + \big( \la \nabla_\tau h(x) - \nabla_\tau h(z) , y-x \ra  \big) \big| \\
 &\leq   \|h\|_{C^{1+s+\alpha}}  |z-x|^{1+s+\alpha} + \|h\|_{C^{1+s+\alpha}}  |z-x|^{s+\alpha} |y-x| \\
 &\leq  2 \|h\|_{C^{1+s+\alpha}}  |z-x|^{s+\alpha}|y-x|.
\end{split}
\end{equation}
Therefore we use the notation for $g_h(y,x)$, which is defined in \eqref{eq:split-by-g}, and the norm   $\|\xi \|_{A(x)}$, which is   defined in \eqref{def:matrixnorm}, and write  
\begin{equation}\label{eq:deri-linear}
\begin{split}
\int_{\Si^n} &\frac{\nabla_i h(y) - \nabla_i h(x) }{|h(y) y - h(x)x|^{n+1+s}}h(y)^{n}  \, d\H_y^n 
\\
&= \int_{\Si^n}  \frac{\nabla_i h(y) - \nabla_i h(x) }{\|y-x\|_{A(x)}^{n+1+s} } h(y)^{n} \, d\H_y^n \\
&+\int_{\Si^n} \frac{\nabla_i h(y) - \nabla_i h(x) }{\|y-x\|_{A(x)}^{n+1+s}} \int_0^1 \frac{d}{d\mu } \left( \frac{\|y-x\|_{A(x)}^{2}}{\|y-x\|_{A(x)}^{2}+ \mu g_h(y,x)}\right)^{\frac{n+1+s}{2}} h(y)^{n}  \,d \mu \, d\H_y^n ,
\end{split}
\end{equation}
where $A(x) = h^2(x)I + \nabla_\tau h(x) \otimes \nabla_\tau h(x)$.  

Recall again the a priori estimates for the height function, i.e., 
\[
\sup_{t \in [0,T)} \|h(\cdot, t))\|_{C^{1+s}} \leq C \quad \text{and}\quad h(x,t) \geq c.
\]
We may thus finally write the equation \eqref{eq:diff-flow} by \eqref{eq:normal-deri}, \eqref{eq:deriv-RHS} and  \eqref{eq:deri-linear}  as 
\begin{equation}
\label{eq:the-pde-1}
\partial_t \nabla_i h =  L_{A}[\nabla_i h]   + R_0(x,t) +R_1(x,t) +   \pa_t h   \, B(x,h, \nabla_\tau h):\nabla_\tau^2 h.
\end{equation}
Here $B$ is a smooth matrix-field and  the linear  operator is defined as
\begin{equation}
\label{def:linear-op}
 L_{A}[u](x) := a_0(x,t) \int_{\Si^n} \frac {u  (y)    -u (x)}{\|y-x\|_{A(x,t)}^{n+1+s}}  b_0(y,t)\, d\H_y^n,
\end{equation}
for symmetric  elliptic matrix field 
\[
A(x,t) = h^2(x,t)I + \nabla_\tau h(x,t) \otimes \nabla_\tau h(x,t)
\]
and the coefficients  $a_0, b_0 \geq c_0$ satisfy  
\begin{equation}
\label{eq:coeff}
\sup_{t < T}\|a_0(\cdot,t)\|_{C^{s}(\Si^n)} \leq C \quad \text{and} \quad \sup_{t < T}\|b_0(\cdot,t)\|_{C^{1}(\Si^n)} \leq C.
\end{equation}

The error terms are 
\begin{equation}
\label{eq:error-R-0}
R_0(x,t) = a_0(x,t)  \int_{\Si^n}\int_0^1   \frac{\nabla_i h(y,t) - \nabla_i h(x,t) }{\| y-x\|_{A(x,t)}^{n+1+s} } \frac{d}{d\mu } \left( \frac{\| y-x\|_{A(x,t)}^{2}}{\| y-x\|_{A(x,t)}^{2} + \mu g_{h}(y,x)}\right)^{\frac{n+1+s}{2}}  b_0(y,t) \, d\mu d\H_y^n
\end{equation}
and
\begin{equation}
\label{eq:error-R}
R_1(x,t) = \sum_{k=1}^{N} a_k(x,h,\nabla_\tau h) \int_{\Si^n} \frac {F_k (y,h(y,t))    -F_k(x,h(x,t))}{|h(y,t)y - h(x,t)x|^{n+1+s}}  b_k(y,h,\nabla_\tau h)\, d\H_y^n, 
\end{equation}
where  $a_k, b_k$ and $F_k$ for $k \geq 1$  are smooth functions.

Next we state  the following technical lemma, which we use later to bound the error term $R_0$ defined in \eqref{eq:error-R-0}. 
\begin{lemma}
\label{lem:bound-G-error}
Assume that $A(\cdot)$ is symmetric elliptic matrix field as defined in \eqref{def:matrix-elliptic} such that $\|A(\cdot)\|_{C^\alpha} \leq C$ for $0 < \alpha < \min\{s,1-s \}$ and assume that $g : \Si^n \times \Si^n \to \R$ satisfies the following two conditions. 
For all $x,y \in \Si^n$  it holds $\| y-x\|_{A(x)}^2 + g(y,x) \geq c| y-x|^2$, for some $c>0$, and 
\[
|g(y,x)| \leq \kappa  |y-x|^{2+s+ \alpha},
\]
 and for all $x,y,z \in \Si^n$ with $|z-x|\leq \frac12 |y-x|$ it holds
\[
|g(y,z) -g(y,x)| \leq \kappa (|z-x|^{s+\alpha} |y-x|^2 + |z-x|^{\alpha} |y-x|^{2+s+\alpha}).
\]
Denote  $G(y,x) =  \frac{d}{d\mu } \left( \frac{\|y-x\|_{A(x)}^2 }{\|y-x\|_{A(x)}^2+ \mu g(y,x)}\right)^{q} $, for $\mu \in (0,1), q \geq 1$, and assume $v_1 \in C^{1}(\Si^n)$, $v_2 \in C(\Si^n)$ and $v_3 \in C^{\alpha}(\Si^n)$. 
Then the function 
\[
\psi(x) = \int_{\Si^n}\big(v_1(y) - v_1(x)\big) v_2(y) v_3(x)G(y,x) K_A(y,x)\, d\H_y^n 
\]
is H\"older continuous and 
\[
\|\psi\|_{C^{\alpha}(\Si^n)}\leq C\kappa  \|v_1\|_{C^{1}(\Si^n)}\|v_2\|_{C^{0}(\Si^n)}\|v_3\|_{C^{\alpha}(\Si^n)}.
\]
\end{lemma} 
\begin{proof}
Let us  first  prove that  it holds for all $y,x \in \Si^n, y\neq x$
\begin{equation}\label{eq:F_h-1}
|G(y,x)|\leq C \kappa  |y-x|^{s+ \alpha}
\end{equation}
and for all $x,y,z \in \Si^n$ with $0 < |z-x|\leq \frac12 |y-x|$ 
\begin{equation}\label{eq:F_h-2}
|G(y,z) -G(y,x)| \leq C\kappa \big( |z-x|^{\alpha}|y-x|^{s+\alpha} + |z-x|^{s+\alpha} \big).
\end{equation}

To this aim we denote  $\Phi(t) = (1+t)^{-q}$,
\[
\tilde G(y,x) = \frac{g(x,y)}{\|y-x\|_{A(x)}^2}
\]
and notice that we may write $G(y,x) = \Phi'(\mu \tilde G(y,x) ) \tilde G(y,x) $. Since $\Phi$ is smooth we deduce that it is enough to show that $\tilde G$ satisfies the conditions \eqref{eq:F_h-1} and \eqref{eq:F_h-2}. We obtain \eqref{eq:F_h-1} immediately  from the ellipticity of $A$ and from the first assumption on $g$ as
\[
|\tilde G(y,x)| \leq C \kappa \frac{ |y-x|^{2+s+ \alpha}}{|y-x|^2}\leq C \kappa  |y-x|^{s+ \alpha}.
\]

To prove \eqref{eq:F_h-2} we fix $x,y,z \in \Si^n$ with $0 < |z-x|\leq \frac12 |y-x|$. Note that then $\frac12 |y-x| \leq |y-z|\leq 2 |y-x|$. We then write 
\[
\begin{split}
|\tilde G(y,z) - \tilde G(y,x)| \leq |g(y,z)| \big|  \|y-z\|_{A(z)}^{-2}   - \|y-x\|_{A(x)}^{-2}  \big| + \frac{|g(x,z) - g(y,z)|}{\|y-x\|_{A(x)}^2 }.
\end{split}
\] 
We may estimate the last term by the second condition on $g$ as
\[
\frac{|g(x,z) - g(y,z)|}{\|y-x\|_{A(x)}^{2}} \leq C \kappa \frac{|z-x|^{s+\alpha} |y-x|^2 + |z-x|^{\alpha}|y-x|^{2+s+\alpha} }{|y-x|^{2}} \leq C \kappa (|z-x|^{s+\alpha}  +  |z-x|^{\alpha} |y-x|^{s+\alpha} ).
\]
For the first term we observe that 
\[
\begin{split}
\big|\|y-z\|_{A(z)}^2 &- \|y-x\|_{A(x)}^2 \big| \\
&=  \big|\la A(z) (y-z), (y-z) \ra  - \la A(x) (y-x), (y-x) \ra \big|\\
&\leq \big| \la (A(z) - A(x)) (y-z), (y-z) \ra \big| + \big| \la A(x)\big((y-z) + (y-x)\big) , (x-z) \ra\big| \\
&\leq C |z-x|^\alpha |y-x|^2 + C   |y-x||z-x|. 
\end{split}
\]
We use the fact that $|z-x|\leq \frac12 |y-x|$ to deduce $|z-x| \leq |z-x|^{s+\alpha}|y-x|^{1-s-\alpha}$. 
Therefore it holds 
\[
\begin{split}
|g(y,z)| \big|  \|y-z\|_{A(z)}^{-2}   - \|y-x\|_{A(x)}^{-2}  \big| &\leq \frac{|g(y,z)|}{|y-x|^4} \big| \|y-z\|_{A(z)}^2 - \|y-x\|_{A(x)}^2 \big|\\
&\leq \frac{C \kappa }{|y-x|^{2-s-\alpha}} (|z-x|^\alpha |y-x|^2 +   |y-x||z-x|)\\
&\leq C \kappa (|z-x|^\alpha |y-x|^{s+\alpha} +  |z-x|^{s+\alpha})
\end{split}
\]
and \eqref{eq:F_h-2} follows.

We proceed by defining  $F: \Si^n \times \Si^n \to \R$ as
\[
F(y,x) : = \big(v_1(y) - v_1(x)\big) v_2(y) v_3(x)G(y,x) 
\]
when $y \neq x$ and $F(x,x)  =0$. We prove that  $F$ satisfies the assumptions of  Lemma \ref{lem:old.4.3}  for 
\[
\tilde \kappa:= \kappa  \|v_1\|_{C^{1}}\|v_2\|_{C^{0}}\|v_3\|_{C^{\alpha}}.
\]
This will then yield the claim.  The first condition follows immediately from \eqref{eq:F_h-1} as 
\[
|F(y,x)| = |v_1(y) - v_1(x)| \, |v_2(y)| \, |v_3(x)|\,  |G(y,x)| \leq  C\tilde \kappa |y-x|^{1+s+\alpha}
\]
for all $x \neq y \in \Si^n$. 

To show the second condition we fix $x,y,z \in \Si^n$ with $0< |z-x|\leq \frac12 |y-x|$. We will in fact show that 
\[
|F(y,z)- F(y,x)|  \leq C \tilde  \kappa \big( |z-x|^{\alpha}|y-x|^{1+ s+\alpha} + |z-x|^{s+\alpha} |y-x| \big),
\]
which is slightly stronger than what we need. To this aim we write
\[
\begin{split}
|F(y,z)- F(y,x)| &=|v_2(y)| | \big(v_1(y) - v_1(z)\big)  v_3(z)G(y,z) - \big(v_1(y) - v_1(x)\big)  v_3(x)G(y,x)  | \\
&\leq |v_2(y)||v_3(z)| |G(y,z) | |v_1(z) - v_1(x)| \\
&\,\,\,\,\,\,\,+  |v_2(y)||v_3(z)| |v_1(y) - v_1(x)|  |G(y,z)- G(y,x) | \\
&\,\,\,\,\,\,\,\,\,\,\,\,\,\,+  |v_2(y)| |v_1(y) - v_1(x)|  |G(y,x) | |v_3(z) -v_3(x) |. 
\end{split}
\]
We estimate the first term  term  on the RHS   by \eqref{eq:F_h-1}  and by $|z-x|\leq  |y-x|$ as
\[
|v_2(y)||v_3(z)| |G(y,z) | |v_1(z) - v_1(x)| \leq C\tilde \kappa |y-x|^{s+\alpha}|z-x|\leq C \tilde \kappa |y-x| |z-x|^{s+\alpha}.
\] 
We estimate the last similarly  
\[
 |v_2(y)| |v_1(y) - v_1(x)|  |G(y,x) | |v_3(z) -v_3(x) | \leq C\tilde \kappa |y-x|^{1+s+\alpha} |z-x|^\alpha.
\]
Finally,   we estimate the second  term  by  \eqref{eq:F_h-2}  as
\[
 |v_2(y)||v_3(z)| |v_1(y) - v_1(x)|  |G(y,z)- G(y,x) | \leq \tilde \kappa \big( |z-x|^{\alpha}|y-x|^{1+ s+\alpha} + |z-x|^{s+\alpha} |y-x| \big).
\]
We thus have the claim by Lemma \ref{lem:old.4.3}. 
\end{proof}

We conclude the section with a technical lemma, which is  analogous to  \cite[Lemma 4.4]{JL}.
\begin{lemma}
\label{lem:old.4.4}
Assume that $A(\cdot)$ is symmetric, elliptic and $\|A\|_{C^\alpha(\Si^n)} \leq C$ for $0 <\alpha < \min \{s,1-s \}$, and  the kernel  $K_A$ is defined in \eqref{def:kernel}. Assume further  $v_1 \in C^{1+s+\alpha}(\Si^n)$, $v_2 \in C^{s+\alpha}(\Si^n)$ and $v_3 \in C^{\alpha}(\Si^n)$. 
Then the function 
\[
\psi(x) = \int_{\Si^n}\big(v_1(y) - v_1(x)\big) v_2(y) v_3(x)K_A(y,x)\, d\H_y^n 
\]
is H\"older continuous and 
\[
\|\psi\|_{C^{\alpha}(\Si^n)}\leq C \|v_1\|_{C^{1+s+\alpha}(\Si^n)}\|v_2\|_{C^{s+\alpha}(\Si^n)}\|v_3\|_{C^{\alpha}(\Si^n)}.
\]

\end{lemma}

\begin{proof}
The claim follows from \cite[Lemma 4.4]{JL} once we show that the function 
\[
\psi(x) = \int_{\Si^n} (y-x)K_A(y,x)\, \, d\H_y^n 
\]
is uniformly $\alpha$-H\"older continuous. Indeed, even if the condition (ii) in Lemma \ref{lem:kernel} is slightly weaker than   \cite[Definition 4.1]{JL}, the claim still follows from the proof of  \cite[Lemma 4.4]{JL}.  Let us first prove the H\"older continuity of $\psi$ under the additional assumption that  at every point $x \in \Si^n$ it holds
\begin{equation}
\label{eq:add-ass} 
\langle A(x) x, \omega \rangle = 0  \qquad \text{for every } \, \omega \quad  \text{orthogonal to } \,  x. 
\end{equation}

To this aim fix $x$ and notice that, by the symmetry of the sphere and by \eqref{eq:add-ass} for all vectors $\omega$ orthogonal to $x$, it holds 
\[
 \int_{\Si^n} \la y-x, \omega \ra K_A(y,x)\, \, d\H_y^n  = 0. 
\]
Therefore 
\[
\psi(x) =  \int_{\Si^n} (y-x)K_A(y,x)\, \, d\H_y^n   = x \int_{\Si^n} \la y-x, x \ra K_A(y,x)\, \, d\H_y^n .
\]
For all $y \in \Si^n$ it holds $ \la y-x, x \ra = -\frac12 |y-x|^2$. Therefore we have 
\[
\psi(x) = - \frac{x}{2} \int_{\Si^n} |y-x|^2 K_A(y,x)\, \, d\H_y^n. 
\] 
Obviously the function $F(y,x) = - \frac{x}{2}|y-x|^2$ satisfies the assumptions of Lemma \ref{lem:old.4.3} and therefore we deduce that $\|\psi\|_{C^\alpha(\Si^n)}\leq C$.   

Let us then consider the general case. We write $A(x) = A_0(x) + A_1(x)$, where 
\[
A_0 = A- (Ax) \otimes x  - x \otimes (Ax) + 2 \la Ax, x\ra (x \otimes x) 
 \]
and $A_1 = A- A_0$. Then $A_0(\cdot)$ is symmetric, elliptic and satisfies the condition \eqref{eq:add-ass}. Moreover,  recalling $ \la y-x, x \ra = -\frac12 |y-x|^2$ and arguing as in \eqref{eq:g_h-1} and \eqref{eq:g_h-2} we observe that $g(y,x)= \langle A_1(x) (y-x), y-x\rangle$ 
 satisfies the following two conditions: for $x,y \in \Si^n$  it holds 
\begin{equation}
\label{eq:add-g-1} 
|g(y,x)| \leq C |y-x|^{3}
\end{equation}
 and for all $x,y,z \in \Si^n$ with $|z-x|\leq \frac12 |y-x|$ it holds
\begin{equation}
\label{eq:add-g-2} 
|g(y,z) -g_\varphi(y,x)| \leq C (|z-x| |y-x|^2  + |z-x|^\alpha |y-x|^3) .
\end{equation}
Since the argument is similar as before, we leave the details for the reader.  At the end, since we may write 
\[
\begin{split}
\psi(x) = &\int_{\Si^n} (y-x)K_A(y,x)\, \, d\H_y^n  = \int_{\Si^n} (y-x)K_{A_0}(y,x)\,  d\H_y^n  \\
&+\int_{\Si^n} (y-x)K_{A_0}(y,x)   \int_0^1    \frac{d}{d\mu } \left( \frac{\|y-x\|_{A_0(x)}^{2}}{\|y-x\|_{A_0(x)}^{2}+ \mu g(y,x)}\right)^{\frac{n+1+s}{2}}\, d \mu \,   d\H_y^n
\end{split}
\]
the conclusion follows from Lemma \ref{lem:bound-G-error}. 
\end{proof}

\section{Schauder estimates for the fractional parabolic equation}
\label{sec:shauder}

In this section we extend Theorem \ref{teo:parabolic} from $\R^n$ to the unit  sphere $\Si^n$. We consider linear operator of the form 
\begin{equation}
\label{def:linear-op-2}
 L_{A}[u](x) :=\int_{\Si^n} \frac {u  (y) -u (x)}{\|y-x\|_{A(x,t)}^{n+1+s}}  a(x,t)b(y,t)\, d\H_y^n,
\end{equation}
where the norm $\|\cdot \|_{A}$ is defined in \eqref{def:matrixnorm},  $A(\cdot, t)$ is symmetric  elliptic matrix field for every $t \in [0,T)$ such that $\sup_{t <T}\|A(\cdot, t)\|_{C^\alpha(\Si^n)} \leq C$ and the coefficients satisfy $a(x,t), b(x,t) \geq c >0$ for all $x \in \Si^n$ and $t \in [0,T)$ and 
\[
\sup_{t < T}\|a(\cdot,t)\|_{C^{s}(\Si^n)} \leq C \quad \text{and} \quad \sup_{t < T}\|b_0(\cdot,t)\|_{C^{1}(\Si^n)} \leq C.
\]
We also assume without mentioning that all  functions are continuous with respect to time.

We prove the following Schauder estimate on the sphere. 
\begin{theorem} \label{thm2:schauder}
Let  $T >0$ and $\alpha$ such that  $0 <\alpha <\min \{s,1-s\}$. Assume that $A(\cdot,t)$ is symmetric and elliptic  such that $\sup_{t <T}\|A(\cdot, t)\|_{C^\alpha(\Si^n)} \leq C$ and let $L_A$ be the operator defined in \eqref{def:linear-op-2}. Let $f:\Si^n \times  [0,T) \to \R$ be such that $\sup_{t <T}\|f(\cdot, t)\|_{C^\alpha(\Si^n)} <\infty$ and let $u : \Si^n \times [0,T) \to \R$ be the solution of the problem
\beq \label{eq:parabolicsphere}
\begin{cases}
\pa_t  u= L_A[u]  +f(x,t) 	\qquad & (x,t) \in \Si^n\times (0, T)
\\
u(x,0)=0                          &x\in \Si^n .
\end{cases}
\eeq
There exists a constant $C_T$ such that
\[
\sup_{t <T} \| u(\cdot, t)\|_{C^{1+s+\alpha}(\Si^n)} \leq C_T\sup_{t <T} \| f(\cdot,t) \|_{C^{\alpha}(\Si^n)} .
\]
\end{theorem}

\begin{proof}
By approximation we may assume that all the functions in the assumptions are smooth. Moreover, we note that we may absorb the coefficient $a(x,t)$ in \eqref{def:linear-op-2} into the matrix field $A(x,t)$ and thus obtain another symmetric elliptic matrix field which we still denote by $A(x,t)$. We also adopt the notation \eqref{def:kernel} for the kernel $K_A$, which means that 
\[
K_A(y,x) = \frac {1}{\|y-x\|_{A(x)}^{n+1+s}} .
\]

First, it follows from the maximum principle that 
\beq \label{eq:schauder-1}
\sup_{t <T}\|u(\cdot, t)\|_{C^0(\Si^n)} \leq  T  \sup_{t <T}\|f (\cdot, t)\|_{C^0(\Si^n)}.
\eeq
Indeed, fix a small $\eps >0$ and define $w_\eps(x,t)= (t+\eps)^{-1} u(x,t)$. Then $w$ is smooth and $w(x,0) = 0$. Assume that $w$ attains its maximum on $\Si^n\times (0,T-\eps]$ at $(x_0,t_0)$ with $t_0 \in (0,T-\eps]$. Then by the maximum principle it holds
\[
0 \leq \pa_t w(x_0,t_0) = \frac{\pa_t u(x_0,t_0)}{t_0+\eps} - \frac{u(x_0,t_0)}{(t_0+\eps)^2}  \quad \text{and} \quad 0\geq L_A[w](x_0,t_0) = \frac{L_A[u](x_0,t_0)}{t_0+\eps}.
\]
The equation \eqref{eq:parabolicsphere} implies 
\[
\frac{u(x_0,t_0)}{t_0+\eps}\leq f(x_0,t_0)  .
\]
By letting $\eps \to 0$ and repeating the argument for $-u$ we obtain \eqref{eq:schauder-1}.

Let us fix $x_0 \in \Si^n$ and by rotating the coordinates we may assume that it is the north pole, i.e.,  $x_0 = e_{n+1}$. Let us first localize the equation around $x_0$. To this aim fix a small $\delta>0$ and choose a smooth cutoff function  $\eta: \R \to [0,1]$ such that $\eta(r)=1$ for $|r|<\delta/2$ and $\eta(r)=0$ for $r\geq\delta$.
In the following we  write $x=(x',x_{n+1}) \in \R^{n+1}$ with $x'\in \R^n$. We denote
\[
v(x,t) = \eta (4|x'|) u(x,t) .
\]
Then clearly $\|v(\cdot, t)\|_{C^\beta(\Si^n)} \leq C \|u(\cdot, t)\|_{C^\beta(\Si^n\cap B_{\delta}(x_0))}$ for $\beta \in [0,2]$. In order to find an equation for $v$ we multiply the equation \eqref{eq:parabolicsphere} by $\eta (4|x'|) $ and obtain
\[
\pa_t  v(x,t)= \eta (4|x'|)  L_A[u](x,t)  + \eta (4|x'|)  \, f(x,t). 
\]
We organize the terms as 
\[
\eta (4|x'|)\big(  u(y,t) -  u(x,t)\big)  = v(y,t) - v(x,t) - u(y,t) ( \eta (4|y'|) - \eta (4|x'|) )
\]
and write 
\[
\begin{split}
 \eta (4|x'|)  L_A[u](x,t)  = &\int_{\Si^n} \eta(|y'|)\big(v(y,t) - v(x,t)\big) K_A(y,x) \, b(y,t)\, d\H_y^n\\
&+\int_{\Si^n} (1-\eta(|y'|) ) \big(v(y,t) - v(x,t)\big) K_A(y,x) \, b(y,t)\, d\H_y^n\\
&- \int_{\Si^n}u(y,t) \big( \eta (4|y'|) - \eta (4|x'|) \big) K_A(y,x) \, b(y,t)\, d\H_y^n .
\end{split}
\]
The last term is H\"older continuous due to Lemma \ref{lem:old.4.4}.  Since $v(x,t) = 0$ for $|x'|\geq \delta/4$, then the Kernel $ (1-\eta(|y'|) ) K_A(y,x)$ in the function 
\[
R_1(x,t) = \int_{\Si^n} (1-\eta(|y'|) ) \big(v(y,t) - v(x,t)\big) K_A(y,x) \, b(y,t)\, d\H_y^n
\]
is non-singular and therefore $\|R_1(\cdot, t)\|_{C^\alpha(\Si^n)} \leq C_\delta \|u(\cdot, t)\|_{C^\alpha(\Si^n)}$, where the constant depends on $\delta>0$. Therefore we deduce that it holds
\begin{equation} \label{eq:schauder-2}
\pa_t  v(x,t)= \int_{\Si^n} \eta(|y'|)\big(v(y,t) - v(x,t)\big) K_A(y,x) \, b(y,t)\, d\H_y^n  + f_1(x,t)
\end{equation}
and $f_1$ satisfies
\begin{equation} \label{eq:schauder-22}
\|f_1(\cdot, t)\|_{C^\alpha(\Si^n)} \leq C_\delta\|u(\cdot, t)\|_{C^{s+\alpha}(\Si^n)} +  C_\delta  \| f(\cdot,t) \|_{C^{\alpha}(\Si^n)}.
\end{equation}

We further  manipulate the leading order term in \eqref{eq:schauder-2} by writing
\begin{equation} \label{eq:schauder-3}
\begin{split}
 \int_{\Si^n} &\eta(|y'|)\big(v(y,t) - v(x,t)\big) K_A(y,x) \, b(y,t)\, d\H_y^n \\
&=  \int_{\Si^n} \eta(|y'|)\big(v(y,t) - v(x,t)\big) K_A(y,x) \, b(e_{n+1},t) \sqrt{1-|y'|^2}\, d\H_y^n +R_2(x,t),
\end{split}
\end{equation}
where
\[
R_2(x,t) :=  \int_{\Si^n} \eta(|y'|)\big(v(y,t) - v(x,t)\big) K_A(y,x) \,(b(y,t)- b(e_{n+1},t) \sqrt{1-|y'|^2})\, d\H_y^n
\]
Let us define $\psi(y,t) = \eta(|y'|) (b(y,t)- b(e_{n+1},t) \sqrt{1-|y'|^2})$. Then at the north pole  it holds $\psi(e_{n+1},t) = 0$. Therefore, since $\psi$ is $C^1$-regular we have
\[
\|\psi(\cdot, t)\|_{C^{s + \alpha}(\Si^n)} \leq C \delta^{1-s-\alpha}. 
\]
Therefore again  Lemma \ref{lem:old.4.4}  yields
\begin{equation} \label{eq:schauder-4}
\|R_2(\cdot, t)\|_{C^\alpha(\Si^n)} \leq C \delta^{1-s-\alpha} \|u\|_{C^{1+s+\alpha}(\Si^n\cap B_{\delta}(x_0))}.
\end{equation}
Note that since $b(e_{n+1},t) $ depends only on time, we may absorb it into the matrix field $A(y,t)$ on the RHS of \eqref{eq:schauder-3} and denote the new matrix field still by $A(y,t)$. 

Let us then consider the RHS  in \eqref{eq:schauder-3}. We write the unit sphere locally as a graph of the function $\varphi(x') = \sqrt{1- |x'|^2}$. Then we may write  $y,x \in \Si^n$ with $|y'|,|x'| < \delta$ as
\[
y-x = \begin{pmatrix} y'-x' \\ \varphi (y')  -\varphi (x') \end{pmatrix} =  \begin{pmatrix} y'-x' \\ \la  \nabla\varphi (x')  , y'-x'\ra \end{pmatrix} + \begin{pmatrix} 0 \\  T_{x'}[\varphi](y') \end{pmatrix}, 
\] 
where $T_{x'}[\varphi](y') = \varphi (y')  -\varphi (x')  - \la  \nabla\varphi (x')  , y'-x'\ra$. Let us define the $(n+1\times n)$ -  matrix field $J_\varphi(x')$ such that  for $\xi' \in \R^n$ it holds
\[
J_\varphi(x')\xi' =  \begin{pmatrix}\xi' \\ \la  \nabla\varphi (x')  , \xi' \ra \end{pmatrix}. 
\]
Then we may write $y-x = J_\varphi(x')(y'-x') + T_{x'}[\varphi](y') e_{n+1}$.  We may write
\[
\la A(x,t) (y-x), (y-x) \ra = \la \tilde A(x,t) (y'-x'), (y'-x') \ra  + g(y,x)
\]
for 
\begin{equation} \label{eq:schauder-5}
\tilde A(x,t)  = J_\varphi(x')^T A(x,t)J_\varphi(x')
\end{equation}
 and 
\[
g(y,x) = 2T_{x'}[\varphi](y') \, \la A(x,t)J_\varphi(x') (y'-x'), e_{n+1}\ra +(T_{x'}[\varphi](y') )^2 \la A(x,t) e_{n+1}, e_{n+1} \ra .
\]
Using $|T_{x'}[\varphi](y')| \leq C|y-x|^2$  and arguing as in \eqref{eq:g_h-1} and \eqref{eq:g_h-2}, we deduce that $g$   satisfies the conditions \eqref{eq:add-g-1} and \eqref{eq:add-g-2}.  
We also note immediately that the new $(n\times n)$ - matrix field $\tilde A(\cdot, t)$ defined in \eqref{eq:schauder-5} is symmetric and elliptic in $\R^n$, and also in this case, with a little abuse of notation, we write $\|\xi\|_{\tilde A(x,t)}^{2}  = \la \tilde A(x,t) \xi,\xi \ra  $ for $\xi \in \R^n$. 
We denote
\[
G_\mu(y,x)  = \frac{d}{d\mu} \left(\frac{ \|y'-x'\|_{\tilde A(x,t)}^2}{\|y'-x'\|_{\tilde A(x,t)}^{2} - \mu g(y,x)}\right)^{\frac{n+1+s}{2}}
\] 
and write the RHS  term in \eqref{eq:schauder-3} as
\begin{equation} \label{eq:schauder-55}
\begin{split}
\int_{\Si^n} &\eta(|y'|)\big(v(y,t) - v(x,t)\big) K_A(y,x) \sqrt{1-|y'|^2}\, d\H_y^n \\
&= \int_{\Si^n} \eta(|y'|)\frac{v(y,t) - v(x,t)}{\|y'-x'\|_{\tilde A(x,t)}^{n+1+s} }\sqrt{1-|y'|^2}\, d\H_y^n + R_3(x,t),
\end{split}
\end{equation}
where 
\[
R_3(x,t) = \int_{\Si^n}\int_0^1 \eta(|y'|)G_\mu(y,x)  \frac{v(y,t) - v(x,t)}{\|y'-x'\|_{\tilde A(x,t)}^{n+1+s}  }\sqrt{1-|y'|^2}\, d \mu \, d\H_y^n.
\]
We may  use Lemma \ref{lem:bound-G-error} to deduce
\begin{equation} \label{eq:schauder-6}
\|R_3(\cdot, t)\|_{C^\alpha(\Si^n)} \leq C \|u\|_{C^1(\Si^n)}. 
\end{equation}

Finally we map the RHS term in \eqref{eq:schauder-55} to $\R^n$,  by slight abuse of notation we write $v(x',t) = v\big((x',\varphi(x')), t\big)$, $\tilde A(x',t) =\tilde A\big((x',\varphi(x')), t\big)$,  and obtain 
\[
\begin{split}
\int_{\Si^n} \eta(|y'|)\frac{v(y,t) - v(x,t)}{\|y'-x'\|_{\tilde A(x,t)}^{n+1+s}   }\sqrt{1-|y'|^2}\, d\H_y^n  &=\int_{\R^n} \eta(|y'|)\frac{v(y',t) - v(x',t)}{\|y'-x'\|_{\tilde A(x',t)}^{n+1+s}   }\, d y' \\
&= \int_{\R^n}\frac{v(y',t) - v(x',t)}{\|y'-x'\|_{\tilde A(x',t)}^{n+1+s}  }\, d y'  + R_4(x',t),
\end{split}
\]
where
\[
R_4(x',t) = -\int_{\R^n}(1- \eta(|y'|)) \frac{v(y',t) - v(x',t)}{\|y'-x'\|_{\tilde A(x',t)}^{n+1+s}  }\, d y' .
\]
Since $v$ has compact support in $B_{\delta/4}$ and $\eta = 1$ in $B_{\delta/2}$, the integral in the definition of $R_4$ is non-singular and  we deduce that 
\[
\|R_4(\cdot, t)\|_{C^{\alpha}(\R^n)} \leq C_\delta \|v(\cdot, t)\|_{C^{\alpha}(\R^n)} \leq C_\delta \|u(\cdot, t)\|_{C^{\alpha}(\Si^n)} . 
\]
We then use  Theorem \ref{teo:parabolic}  and  \eqref{eq:schauder-2}, \eqref{eq:schauder-22}, \eqref{eq:schauder-3}, \eqref{eq:schauder-4}, \eqref{eq:schauder-55}, \eqref{eq:schauder-6}  and the above to conclude that $v$ is a solution of the equation
\[
\pa_t v = L_{\tilde A}(v) +  f_2
\]
and satisfies 
\[
\begin{split}
\sup_{t <T} &\|v(\cdot, t)\|_{C^{1+s+\alpha}(\R^n)} \leq C \sup_{t <T}  \|f_2(\cdot, t)\|_{C^\alpha(\R^n)}\\
 &\leq \sup_{t <T} \big(C_\delta \|u(\cdot, t)\|_{C^{1}(\Si^n)} + C \delta^{1-s-\alpha} \|u(\cdot, t)\|_{C^{1+s+\alpha}(\Si^n\cap B_{\delta}(x_0))}+  C_\delta  \| f(\cdot,t) \|_{C^{\alpha}(\Si^n)}\big).
\end{split}
\]
We first observe that 
\[
  \sup_{t <T} \|u(\cdot, t)\|_{C^{1+s+\alpha}(\Si^n\cap B_{\delta/8}(x_0))} \leq C \sup_{t <T} \|v(\cdot, t)\|_{C^{1+s+\alpha}(\R^n)}.
\]
We repeat the argument in balls $B_{\delta}(x_i)$ which cover the sphere, and use a standard covering argument to conclude that 
\[
 \sup_{t <T} \|u(\cdot, t)\|_{C^{1+s+\alpha}(\Si^n)} \leq  \sup_{t <T} \big(C_\delta \|u(\cdot, t)\|_{C^{1}(\Si^n)} + C \delta^{1-s-\alpha} \|u(\cdot, t)\|_{C^{1+s+\alpha}(\Si^n)}+  C_\delta  \| f(\cdot,t) \|_{C^{\alpha}(\Si^n)}\big).
\]
Choosing $\delta>0$ so  small that $C \delta^{1-s-\alpha} \leq \frac14$ implies
\[
 \sup_{t <T} \|u(\cdot, t)\|_{C^{1+s+\alpha}(\Si^n)} \leq  \sup_{t <T} \big(C_\delta \|u(\cdot, t)\|_{C^{1}(\Si^n)} + \tfrac14 \|u(\cdot, t)\|_{C^{1+s+\alpha}(\Si^n)}+  C_\delta  \| f(\cdot,t) \|_{C^{\alpha}(\Si^n)}\big).
\]
We use the interpolation inequality  \eqref{def:interpolation-eucl} and Young's inequality 
\[
\|u(\cdot, t)\|_{C^{1}(\Si^n)}  \leq C \|u(\cdot, t)\|_{C^{1+s+\alpha}(\Si^n)}^{\theta} \|u(\cdot, t)\|_{C^{0}(\Si^n)}^{1-\theta}\leq \eps  \|u(\cdot, t)\|_{C^{1+s+\alpha}(\Si^n)} +C_\eps \|u(\cdot, t)\|_{C^{0}(\Si^n)}.
\]
This yields
\[
 \sup_{t <T} \|u(\cdot, t)\|_{C^{1+s+\alpha}(\Si^n)} \leq  \sup_{t <T}\big(( \eps C_\delta+ \tfrac14) \|u(\cdot, t)\|_{C^{1+s+\alpha}(\Si^n)}+  C_\eps \|u(\cdot, t)\|_{C^{0}(\Si^n)} +  C_\delta  \| f(\cdot,t) \|_{C^{\alpha}(\Si^n)}\big).
\]
We then choose $\eps >0$  so small that $\eps C_\delta \leq \frac14$ and have 
\[
 \sup_{t <T} \|u(\cdot, t)\|_{C^{1+s+\alpha}(\Si^n)}  \leq C_{\delta, \eps} \sup_{t <T} (\|u(\cdot, t)\|_{C^{0}(\Si^n)} +   \| f(\cdot,t) \|_{C^{\alpha}(\Si^n)}).
\] 
The claim then follows from \eqref{eq:schauder-1}.
\end{proof}

\section{$C^{2+\alpha}$-estimates for the flow}

%

\begin{proof}[Proof of the Main Theorem]
By the result \cite{JL} the equation  \eqref{eq:the-flow} has $C^\infty$-solution  for a time interval $(0,T_0)$ and we denote it by $(E_t)_{t \in (0,T_0)}$. We assume that $T_0$ is the maximal time of existence and our aim is to show that $T_0 = \infty$, i.e., the flow does not develop singularities. We argue by contradiction and assume $T_0 < \infty$.  

By the result \cite{CNR}  the sets $E_t$ are convex and  by \cite{CSV2} they satisfy
\begin{equation}\label{eq:CSV-1}
\sup_{t < T_0} \|H_{E_t}^s\|_{L^\infty(\pa E_t)} \leq C
\end{equation}
where the constant is independent of $T_0$. Moreover again by  \cite{CSV2} there are points $x_t \in \R^{n+1}$ such that  
\begin{equation}\label{eq:CSV-2}
B_r(x_t) \subset E_t \subset B_R(x_t)
\end{equation}
for $0 < r < R$, which are independent of $T_0$. By translation we may assume that $x_{T_0} = 0$. By \eqref{eq:CSV-1} it holds that the normal velocity of the flow is bounded
\[
\sup_{t < T_0} \|V_t\|_{L^\infty(\pa E_t)} \leq C .
\]
This implies that there is $ t_0 \in(0,T_0)$ such that $B_{r/2} \subset E_t \subset B_{2R}$ for all $t \in [t_0,T_0]$. By restarting the flow at $t_0$ and by reparametrizing the time, we may assume that $t_0 = 0$ and  denote the maximal time of existence by $T$. 

From these results we deduce that we may parametrize the sets $E_t$ via the height function over the unit sphere, i.e., there are $h(\cdot, t) : \Si^n \to \R$ such that $h(x,t) \geq c_0 >0$ and   
\[
\pa E_t = \{ h(x,t) x : x \in \Si^n\} \qquad \text{for all }  t \in [0, T). 
\]
By \eqref{eq:CSV-1} and Proposition \ref{prop.curvaturebound} it holds 
\begin{equation}\label{eq:apriori-C-1}
\sup_{t < T}  \|h(\cdot, t)\|_{C^{1+s}(\Si^n)} \leq C.
\end{equation}
As we mentioned in the introduction, our goal is to improve the estimate \eqref{eq:apriori-C-1} to 
\begin{equation}\label{eq:apriori-C-2}
\sup_{t  < T}  \|h(\cdot, t)\|_{C^{2+s+\alpha}(\Si^n)} \leq C,
\end{equation}
where $\alpha>0$ is such that $\alpha < \min\{\tfrac{s}{4}, 1-s\}$. The bound $\alpha < \tfrac{s}{4}$ is technical and the reason will be clear later. The authors in \cite{CN} (see also \cite{CSV2}) already  point out that the estimate \eqref{eq:apriori-C-2} implies  that the  maximal time of existence is infinity. Let us briefly recall why it  is so. If we have \eqref{eq:apriori-C-2} and if we assume $T <\infty$,  then by continuity, the solution of   the equation \eqref{eq:the-flow} exists up to the endpoint $[0,T]$. Since the set $E_{T}$ is $C^{2}$-regular,  by  \cite{JL}  we may  restart the flow and obtain by the semigroup property that the solution exists for the time interval $[0,T+\delta]$ for some $\delta >0$. This contradicts the fact that $T$  is the maximal time of existence. Hence, once we have \eqref{eq:apriori-C-2} we obtain that the flow exists globally in time. The fact that the  flow $(E_t)_{t \geq 0}$ converges to some ball $B_1(x_0)$ exponentially fast follows from \cite[Corollary 3.5]{CN}.  Hence, we need to show \eqref{eq:apriori-C-2}.

By \eqref{eq:the-pde-1} we may write the equation \eqref{eq:diff-flow} using the height functions as 
\[
\partial_t \nabla_i h =  L_{A}[\nabla_i h]   + R_0(x,t) +R_1(x,t) +   \pa_t h   \, B(x,h, \nabla_\tau h):\nabla_\tau^2 h
\]
with initial condition $h(\cdot, 0) = h_0$, where $h_0 \in C^\infty(\Si^n)$. Here $R_0$ is defined in \eqref{eq:error-R-0} and $R_1$ in  \eqref{eq:error-R}. We define  $u(x,t) := \nabla_i h(x,t) - \nabla_i h_0(x)$ and deduce that it is a solution of 
\[
\partial_t  u  =  L_{A}[u] + f_i(x,t)  + L_{A}[\nabla_i h_0]
\]
with $u(x,0) = 0$ for
\[
f_i(x,t)  =  R_0(x,t) +R_1(x,t) +   \pa_t h   \, B(x,h, \nabla_\tau h):\nabla_\tau^2 h.
\]
Theorem \ref{thm2:schauder} implies 
\begin{equation} \label{eq:C-2-estimate-1}
\sup_{t < T} \|u(\cdot, t)\|_{C^{1+s + \alpha}(\Si^n)} \leq C (\sup_{t < T} \|f_i(\cdot, t)\|_{C^{\alpha}(\Si^n)} +  \|L_{A}[\nabla_i h_0] \|_{C^{\alpha}(\Si^n)}).
\end{equation}
We use this for every direction $e_i$ and obtain 
\[
\sup_{t \leq T_0} \|h(\cdot, t)\|_{C^{2+s + \alpha}(\Si^n)} \leq C\sup_{t \leq T} \max_{i=1,\dots, n+1}\|f_i(\cdot, t)\|_{C^{\alpha}} +C( \|h_0\|_{C^{2+s + \alpha}} + \|L_{A}[\nabla_i h_0] \|_{C^{\alpha}}).
\]
We are then left to estimate the RHS of \eqref{eq:C-2-estimate-1}, i.e., for every $i = 1,\dots, n+1$ we need to bound
\begin{equation} \label{eq:C-2-estimate-2}
\|L_{A}[\nabla_i h_0] \|_{C^{\alpha}(\Si^n)} + \|R_0(\cdot, t)\|_{C^{\alpha}(\Si^n)} +   \|R_1(\cdot, t)\|_{C^{\alpha}(\Si^n)}  +  \|\pa_t h   \, B(x,h, \nabla_\tau h):\nabla_\tau^2 h\|_{C^{\alpha}(\Si^n)} .
\end{equation}
The first term is easy to bound,  since $h_0 \in C^{\infty}$. Indeed,  by Lemma \ref{lem:old.4.4} it holds $\|L_{A}[\nabla_i h_0] \|_{C^{\alpha}(\Si^n)} \leq C \|h_0\|_{C^{2+s + \alpha}} \leq C$. 

For the second term on the RHS of \eqref{eq:C-2-estimate-2} we recall that the term $R_0(\cdot, t)$ is given by  \eqref{eq:error-R-0}. We denote 
\begin{equation} \label{def:G-tau}
G_\mu(x,y) = \frac{d}{d\mu } \left( \frac{\| y-x\|_{A(x,t)}^{2}}{\| y-x\|_{A(x,t)}^{2} + \mu g_{h}(y,x)}\right)^{\frac{n+1+s}{2}}
\end{equation}
and recall that by \eqref{eq:g_h-1} and \eqref{eq:g_h-2} the function $g_{h}$ satisfies the conditions of Lemma \ref{lem:bound-G-error}  with $\kappa = C \|h\|_{C^{1+s+\alpha}}$. Since we may write  
\[
R_0(\cdot, t) = a_0(x,t)  \int_{\Si^n}\int_0^1   \frac{\nabla_i h(y,t) - \nabla_i h(x,t) }{\| y-x\|_{A(x,t)}^{n+1+s} } G_\mu(x,y)   b_0(y,t) \, d\mu\, d\H_y^n
\]
we may use  Lemma \ref{lem:bound-G-error} with $v_1 = \nabla_i h(\cdot, t), v_2 = b_0(\cdot,t)$ and $v_3 = a_0(\cdot,t) $  and recall that the coefficients $a_0$ and $b_0$ satisfy the conditions \eqref{eq:coeff}  to  deduce
\begin{equation} \label{eq:R-0-estimate}
\sup_{t \leq T} \|R_0(\cdot, t)\|_{C^{\alpha}} \leq \sup_{t \leq T} C \|h(\cdot, t)\|_{C^{1+s+\alpha}} \|h(\cdot, t)\|_{C^{2}} . 
\end{equation}

To  treat the term $\|R_1(\cdot, t)\|_{C^{\alpha}(\Si^n)}$ in \eqref{eq:C-2-estimate-2} we recall that it is of the form \eqref{eq:error-R}. We fix $t$ and write $F_k(x)  = F(y, h(x,t)), a_k(x) = a_k(x,h, \nabla_\tau h)$ and  $b_k(y) = b_k(y,h, \nabla_\tau h)$ for short, and note that then 
\begin{equation} \label{eq:R-estimate1}
 \|F_k\|_{C^{k+ \beta}} \leq C \|h\|_{C^{k+\beta}} , \quad   \|a_k\|_{C^{k+ \beta}} \leq C \|h\|_{C^{k+1+\beta}} \quad \text{and} \quad  \|b_k\|_{C^{k+ \beta}} \leq C \|h\|_{C^{k+1+\beta}}
\end{equation}
for all $k = 0,1,2, \dots$ and $\beta \in [0,1]$. We use  the  argument  from \eqref{eq:deri-linear} and write  for $A(x,t) = h(x,t)^2I + \nabla_\tau h(x,t) \otimes \nabla_\tau h(x,t)$
\[
\begin{split}
\int_{\Si^n} &a_k(x) b_k(y) \frac{F_k(y) -F_k(x) }{|h(y,t) y - h(x,t)x|^{n+1+s}} \, d\H_y^n 
\\
&= \overbrace{\int_{\Si^n} a_k(x) b_k(y) \frac{F_k(y) -F_k(x) }{\| y-x\|_{A(x,t)}^{n+1+s} } \, d\H_y^n}^{=: \rho_1} \\
&+\int_{\Si^n} a_k(x) b_k(y) \frac{F_k(y) -F_k(x) }{\| y-x\|_{A(x,t)}^{n+1+s} } \int_0^1 \frac{d}{d\mu } \left( \frac{\| y-x\|_{A(x,t)}^{2}}{\| y-x\|_{A(x,t)}^{2} + \mu g_h(y,x)}\right)^{\frac{n+1+s}{2}}\, d \mu \, d\H_y^n = \rho_1(x) + \rho_2(x),
\end{split}
\]
where $g_h$ is given by \eqref{eq:split-by-g}.   To estimate the H\"older norm of the first term $\rho_1$ on the RHS above, we use Lemma \ref{lem:old.4.4},   \eqref{eq:apriori-C-1},\eqref{eq:R-estimate1}   and deduce that we may bound its $C^{\alpha}$-norm by 
\[
\|\rho_1\|_{C^\alpha}\leq  C \|F_k\|_{C^{1+s+\alpha}} \|b_k\|_{C^{s+\alpha}}\|a_k\|_{C^{\alpha}}\leq C \|h\|_{C^{1+s+\alpha}} \|h\|_{C^{1+s+\alpha}}\|h\|_{C^{1+\alpha}}\leq C \|h\|_{C^{1+s+\alpha}}^2.  
\]
To estimate the H\"older norm of the second term $\rho_2$, we argue exactly as in \eqref{eq:R-0-estimate}, define  $G_\mu$ as in \eqref{def:G-tau} and recall that it satisfies the conditions of Lemma \ref{lem:bound-G-error} with  $\kappa = C \|h\|_{C^{1+s+\alpha}}$.  We may then estimate  $\|\rho_2\|_{C^\alpha}$ by Lemma \ref{lem:bound-G-error}   with $v_1 = F_k, v_2 = b_0$ and $v_3 = a_0$ and by \eqref{eq:R-estimate1} to  deduce 
\[
\|\rho_2\|_{C^\alpha}\leq C \|h\|_{C^{1+s+\alpha}}\|F_k\|_{C^1} \|b_0\|_{C^0} \|a_0\|_{C^\alpha} \leq C \|h\|_{C^{1+s+\alpha}}  \|h\|_{C^{1}}^2 \|h\|_{C^{1+\alpha}} \leq  C \|h\|_{C^{1+s+\alpha}}. 
\]
In conclusion we have 
\begin{equation} \label{eq:R-estimate}
\sup_{t \leq T} \|R_1(\cdot, t)\|_{C^{\alpha}} \leq \sup_{t \leq T} C \|h(\cdot, t)\|_{C^{1+s+\alpha}}^2 \leq \sup_{t \leq T} C \|h(\cdot, t)\|_{C^{1+s+\alpha}} \|h(\cdot, t)\|_{C^{2}} . 
\end{equation}

We are left with the last term in \eqref{eq:C-2-estimate-2}.  We first  use \eqref{eq:apriori-C-1}  and obtain 
\begin{equation} \label{eq:pa-t-estimate-0}
\|\pa_t h   \, B(x,h, \nabla_\tau h):\nabla_\tau^2 h\|_{C^{\alpha}} \leq C \|\pa_t h   \nabla_\tau^2 h\|_{C^{\alpha}} \leq C \|\pa_t h \|_{C^{\alpha}}  \| h\|_{C^{2+\alpha}}.\end{equation}
We claim that  when $\alpha \leq \tfrac12$ it holds
\begin{equation} \label{eq:pa-t-estimate}
\sup_{t \leq T} \|\pa_t h \|_{C^{\alpha}}  \leq \sup_{t \leq T} C(1+ \|h\|_{C^{2+\alpha}}^{2\alpha} ) . 
\end{equation}
Note that the assumption $\alpha < \min\{\tfrac{s}{4},1-s\}$ implies  $\alpha < \tfrac12$.

To this aim we recall that the normal velocity of the flow is bounded, i.e., 
\[
\sup_{t \leq T} \|V_t\|_{L^\infty(\pa E_t)} \leq C .
\]  
We then use the parametrization of the normal velocity \eqref{eq:normal},  the a priori estimates $h \geq c_0$  and \eqref{eq:apriori-C-1}  to deduce that 
\[
\sup_{t \leq T} \|\pa_t h \|_{C^0} \leq C .
\]
We use  the interpolation inequality from Proposition \ref{label:interpolation} to estimate
\begin{equation} \label{eq:pa-t-estimate2}
\|\pa_t h \|_{C^{\alpha}} \leq C \|\pa_t h \|_{C^{1}}^{\alpha}\|\pa_t h \|_{C^{0}}^{1-\alpha} \leq C  \|\pa_t h \|_{C^{1}}^{\alpha}.
\end{equation}
Recall that $\nabla_i h$ is a solution of the equation 
\[
\partial_t  \nabla_i h  =  L_{A}[\nabla_i h] + f_i(x,t).  
\]
Lemma \ref{lem:old.4.4} yields $\|L_{A}[\nabla_i h] \|_{C^{\alpha}} \leq C \|h\|_{C^{2+s+\alpha}}$ and thus  by Theorem \ref{thm2:schauder}  we have 
\[
\sup_{t \leq T} \|\pa_t  h \|_{C^{1}} \leq \sup_{t \leq T} \|\pa_t  h \|_{C^{1+\alpha}} \leq \sup_{t \leq T} \max_{i = 1, \dots, n+1} C(1+\|f_i(\cdot, t)\|_{C^{\alpha}}).
\]
We use this,  \eqref{eq:R-0-estimate}, \eqref{eq:R-estimate}, \eqref{eq:pa-t-estimate-0}, \eqref{eq:pa-t-estimate2}  and Young's inequality to estimate 
\[
\begin{split}
\sup_{t \leq T} \|\pa_t  h \|_{C^{1}} &\leq C \sup_{t \leq T} \big(1+  \|h\|_{C^{1+s+\alpha}} \|h\|_{C^{2}} + \|\pa_t h \|_{C^{\alpha}}  \| h\|_{C^{2+\alpha}}\big)\\
&\leq  \sup_{t \leq T} \big(C+  C\|h\|_{C^{2}}^2 +C\eps  \|\pa_t h \|_{C^{\alpha}}^{\frac{1}{\alpha}}  + C_\eps \| h\|_{C^{2+\alpha}}^{\frac{1}{1-\alpha}}\big)\\
&\leq \sup_{t \leq T} \big(C+  C\|h\|_{C^{2}}^2 +C \eps  \|\pa_t h \|_{C^{1}}+ C_\eps \| h\|_{C^{2+\alpha}}^{\frac{1}{1-\alpha}}\big).
\end{split}
\] 
By choosing $\eps$ so small that $C\eps \leq \tfrac12$ we have
\[
\sup_{t \leq T} \|\pa_t  h \|_{C^{1}} \leq C \sup_{t \leq T} \big(1+  \|h\|_{C^{2}}^2 +  \| h\|_{C^{2+\alpha}}^{\frac{1}{1-\alpha}}\big).
\]
When $\alpha \leq \frac12$ it holds $\frac{1}{1-\alpha} \leq 2$. Therefore we obtain \eqref{eq:pa-t-estimate}  from  \eqref{eq:pa-t-estimate2} and from  the above.

Now it follows from \eqref{eq:pa-t-estimate-0} and  \eqref{eq:pa-t-estimate}  that 
\[
\sup_{t \leq T}\|\pa_t h   \, B(x,h, \nabla_\tau h):\nabla_\tau^2 h\|_{C^{\alpha}} \leq \sup_{t \leq T} C(1+  \| h\|_{C^{2+\alpha}}^{1+2\alpha}) .
\]
Therefore we obtain by \eqref{eq:C-2-estimate-1}, \eqref{eq:R-0-estimate}, \eqref{eq:R-estimate} and from the above that 
\[
\sup_{t \leq T_0} \|h(\cdot, t)\|_{C^{2+s + \alpha}} \leq  \sup_{t \leq T} C\big(1+  \|h(\cdot, t)\|_{C^{1+s+\alpha}} \|h(\cdot, t)\|_{C^{2}}   +  \| h(\cdot, t)\|_{C^{2+\alpha}}^{1+2\alpha}\big) .
\]  
We use the interpolation inequality from Proposition \ref{label:interpolation} to estimate first 
\[ 
 \|h(\cdot, t)\|_{C^{1+s+\alpha}} \leq C  \|h(\cdot, t)\|_{C^{2+s+\alpha}}^{\frac{\alpha}{1+\alpha}} \|h(\cdot, t)\|_{C^{1+s}}^{\frac{1}{1+\alpha}}  
\]
and then 
\[
 \|h(\cdot, t)\|_{C^{2}}\leq  \|h(\cdot, t)\|_{C^{2+\alpha}} \leq C  \|h(\cdot, t)\|_{C^{2+s+\alpha}}^{\frac{1-s+\alpha}{1+\alpha}} \|h(\cdot, t)\|_{C^{1+s}}^{\frac{s}{1+\alpha}} . 
\]
Recall that \eqref{eq:C-2-estimate-1} implies $\sup_{t < T}\|h(\cdot, t)\|_{C^{1+s}}\leq C $. Therefore we obtain by combining the above inequalities
\begin{equation} \label{eq:pa-t-estimate3}
\sup_{t < T} \|h(\cdot, t)\|_{C^{2+s + \alpha}} \leq  \sup_{t <T} C\big(1+  \|h(\cdot, t)\|_{C^{2+s+\alpha}}^{\frac{1-s+2\alpha}{1+\alpha}}  + \|h(\cdot, t)\|_{C^{2+s+\alpha}}^{\frac{1-s+\alpha}{1+\alpha}(1+2 \alpha)}\big) .
\end{equation}
Recall that we choose $\alpha < \frac{s}{4}$. This implies 
\[
\frac{1-s+2\alpha}{1+\alpha} < \frac{1}{1+\alpha} <1 \quad \text{and} \quad \frac{1-s+\alpha}{1+\alpha}(1+2 \alpha) < \frac{1}{1+\alpha} <1. 
\]
Thefore we obtain from \eqref{eq:pa-t-estimate3} that 
\[
\sup_{t < T} \|h(\cdot, t)\|_{C^{2+s + \alpha}} \leq  C\big(1+  \sup_{t <T} \|h(\cdot, t)\|_{C^{2+s+\alpha}}^{\frac{1}{1+\alpha} } \big)
\]
which implies 
\[
\sup_{t < T} \|h(\cdot, t)\|_{C^{2+s + \alpha}} \leq  C
\]
and the claim  \eqref{eq:apriori-C-2} follows. 
 \end{proof}

\appendix
\section {}

In the appendix we give a self-contained proof  of Theorem \ref{teo:parabolic}. Recall that we consider linear operator 
\[
L_A[u](x)  = \int_{\R^n} \frac{u(y+x)-u(x)}{\la A(x,t) y, y\ra^{\frac{n+1+s}{2}}} dy =  \int_{\R^n} \frac{u(y)-u(x)}{\la A(x,t) (y-x), (y-x)\ra^{\frac{n+1+s}{2}}} dy .
\]
In the following, we denote by $\F [u](\xi)$  the Fourier transform of $u$ evaluated at a point $\xi$
\[
\F [u](\xi)= \int_{\R^n} u(x) e^{-2\pi i\langle\xi ,x\rangle}\, dx
\]
and denote by $\F^{-1}$ the inverse Fourier transform, i.e., 
\[
\F^{-1}[v](x) =  \int_{\R^n} v(\xi) e^{2 \pi i\langle\xi ,x\rangle}\, d\xi.
\] 
The proof of Theorem \ref{teo:parabolic} follows from  a  small perturbation argument. We begin with an easy lemma.
\begin{lemma} \label{lem:constant}
Let $A(t)$ be elliptic and symmetric matrix which is constant w.r.t $x$. Then
\beq \label{eq:formula1}
\F \big(L_{A}[u]\big)(\xi) =  -a(\xi,t) |\xi|^{1+s}\F[u](\xi)
\eeq
where $a(\xi,t)$ is  $0$-homogeneous function w.r.t to $\xi$ given by 
\[
a(\xi,t)= \int_{\Si^{n-1}}\int_0^\infty \frac{1- \cos(2 \pi r\langle \frac{\xi}{|\xi|},\omega\rangle )}{ r^{2+s}\langle A(t)\omega,\omega\rangle^\frac{n+1+s}{2}}\,dr\, d\H^{n-1}_\omega.
\]
Moreover,  it holds 
\beq \label{eq:a0}
\frac{1}{C} \leq a(\xi,t) \leq C 
\;\;\;\text{and}\;\;\;
|\pa ^{\beta} a(\xi,t)|\leq C_{|\beta|} \frac{1}{|\xi|^{|\beta|}}  \qquad \text{for } \xi \not = 0 \text{ and } t\in (0,T).
\eeq
\end{lemma}
\begin{proof}
Since $A(t)$ is constant with respect to the space  and we are applying Fourier transform with respect to space, we drop the time dependence.  
Setting $u_y(x):= u(y+x)$, we have $\F[u_y](\xi)= e^{2 \pi i\langle \xi, y\rangle} \F[u] (\xi)$. 
By linearity of the Fourier transform, we get
\[\begin{split}
\F[ L_{A} u](\xi)&=
\F\left[    \int_{\R^n} \frac{u(y+x)- u(x)}{\langle A\,y,y\rangle^\frac{n+1+s}{2}}\,dy       \right](\xi) =  \int_{\R^n} \frac{\F[u_y](\xi)- \F[u](\xi)}{\langle A\,y,y\rangle^\frac{n+1+s}{2}}\,dy  \\
\\&=- \F[u](\xi)   \int_{\R^n} \frac{1- \cos(2 \pi \langle \xi,y\rangle)}{\langle A\,y,y\rangle^\frac{n+1+s}{2}}\,dy 
\\
&=-|\xi|^{1+s}\F[u](\xi) \int_{\Si^{n-1}}\int_0^\infty \frac{1- \cos(2 \pi r\langle \frac{\xi}{|\xi|},\omega\rangle )}{ r^{2+s}\langle A\omega,\omega\rangle^\frac{n+1+s}{2}}\,dr\, d\H^{n-1}_\omega.
\end{split} 
\]
From this formula and the fact that $a(\xi)$ is $0$-homogeneous it is immediate to obtain  the bounds \eqref{eq:a0}
\end{proof}

In order to prove Theorem \ref{teo:parabolic}, we first consider the case of coefficients which are constant in space. 
To that aim we recall that we may characterize the H\"older continuity   by its Fourier transform using the Hardy-Littlewood decomposition. Let  
$\eta \in C_0^\infty(\R^n)$be  such that $0 \leq \eta \leq 1$, $\text{supp} \, \eta \subset  \bar{B}_2$,  and $\eta(\xi) = 1$  for $\xi \in B_1$. Define then $\delta(\xi) = \eta(\xi) - \eta(2\xi)$. Then the functions $\delta(2^{-j}\xi )$ form a partition of unity, i.e.
\[
1 = \sum_{j= - \infty}^\infty \delta(2^{-j} \xi), \qquad \text{for }\,   \xi \neq 0.
\]
Next we define  $\Psi : \R^n \to \R$ via its Fourier transform
\[
\F[\Psi](\xi) = \delta(\xi) = \eta(\xi) - \eta(2\xi).
\]
We write $\Psi_t(x) = t^{-n}\Psi(x/t)$. Then we have by scaling  $\F[\Psi_{2^{-j}}](\xi) = \delta(2^{-j}\xi)$. Finally we define the operator $\Delta_j$ by convolution
\[
\Delta_j(f) := f * \Psi_{2^{-j}}. 
\]
Note that since $\F [u * v] = \F [u] \cdot \F [v]$ we may write 
\[
\Delta_j(f)   = \mathcal{F}^{-1}[ \F [f] \cdot \delta(2^{-j}\xi)].
\]
We recall that when $\gamma >0$ is not an integer it holds  (see e.g. \cite{Tri})
\beq \label{eq:fourierholder}
\frac{1}{C} \|f\|_{C^\gamma(\R^n)} \leq \sup_{j \geq 1} \, 2^{j\gamma} \|\Delta_j(f)\|_{L^\infty} \leq C \|f\|_{C^\gamma(\R^n)}. 
\eeq
We assume that $A(t)$ is an elliptic and symmetric  matrix  field which is continuous w.r.t to time. 


\begin{theorem} \label{thm:timedependent}
Let $u\in C^{1+s+\alpha}(\R^n)$ be a solution of 
\beq
\label {eq:timedependent}
\begin{cases}
\pa_t u =  L_{A(t)} [u]+ f(x,t),
\\
u(x,0)=0
\end {cases}
\eeq 
where the matrix field is symmetric, elliptic, continuous w.r.t. time and constant in space. For  $\alpha< \min\{ s, 1-s\}$ it holds  
\[
\sup_{t<T} \|u(\cdot, t)\|_{C^{1+s+\alpha}} \leq C_T \sup_{t<T} \| f(\cdot, t)\|_{C^{\alpha}},
\]
where the constant $C$ depends only on $\alpha, s,n,T$ and on the ellipticity constants of $A(\cdot)$.
\end{theorem}
\begin{proof}
Applying Fourier transform to  the equation \eqref{eq:parabolic} and using Lemma \ref{lem:constant} we obtain 
\[
\pa_t \F [u](\xi,t) = -a(\xi,t)|\xi|^{1+s} \F [u](\xi,t) + \F [f](\xi,t).
\]
Multiplying the above  by $e^{  |\xi|^{1+s}   \int_{0}^{t}a(\xi,\tau)\, d\tau}$ we
 have 
\[
\pa_t \big(\F [u](\xi,t) e^{ g(\xi,t)|\xi|^{1+s} } \big) =  \F [f](\xi,t)e^{|\xi|^{1+s}  g(\xi,t)}
\]
where we have set $g(\xi,t)=\int_{0}^t a(\xi,\tau)\, d\tau$.
Integrate over  $(0,t)$, recall that $u(x,0) = 0$, and get
\[
\F [u](\xi,t) e^{g(t,\xi)|\xi|^{1+s}}  = \int_0^t  \F [f](\xi,\tau)e^{g(\tau,\xi) |\xi|^{1+s}}\, d \tau.
\]
Thus we may write the solution  as
\[
\F [u](\xi,t) = \int_0^t  \F [f](\xi,\tau)e^{-(g(\xi,t)-g(\xi,\tau)) |\xi|^{1+s}}\, d \tau . 
\]

We need to estimate $\|\Delta_j(u(\cdot,t))\|_{L^\infty}  $. Since $\delta $ is a cutoff function it is enough to estimate 
\[
\|\mathcal{F}^{-1}\big(\F [u](\xi,t)\cdot \delta^2(2^{-j}\xi) \big)\|_{L^\infty} \leq  \int_0^t \|\mathcal{F}^{-1}\big( e^{-(g(\xi,t)-g(\xi,\tau)) |\xi|^{1+s}}  \F [f](\xi,\tau) \delta^2(2^{-j}\xi) \big)\|_{L^\infty}  \, d \tau.
\]
 Since the product becomes a convolution we have
\[
\begin{split}
\|\mathcal{F}^{-1}&\big( e^{-(g(\xi,t)-g(\xi,\tau)) |\xi|^{1+s}} \F [f](\xi,\tau) \delta^2(2^{-j}\xi) \big)\|_{L^\infty}\\
&= \|\mathcal{F}^{-1}\big( e^{-(g(\xi,t)-g(\xi,\tau)) |\xi|^{1+s}}\delta(2^{-j}\xi) \big) \, * \, \mathcal{F}^{-1} \big( \F [f](\xi,\tau) \delta(2^{-j}\xi) \big)\|_{L^\infty}\\
&\leq \|\mathcal{F}^{-1}\big( e^{-(g(\xi,t)-g(\xi,\tau)) |\xi|^{1+s}} \delta(2^{-j}\xi) \big)\|_{L^1}\, \|\mathcal{F}^{-1}\big(   \F [f](\xi,\tau) \delta(2^{-j}\xi) \big)\|_{L^\infty}.
\end{split}
\]
Since $f(\cdot,t) \in C^\alpha(\R^n)$  for every $\tau <t$ we have by \eqref{eq:fourierholder} 
\[
\|\mathcal{F}^{-1}\big(   \F [f](\xi,\tau) \delta(2^{-j}\xi) \big)\|_{L^\infty} \leq C 2^{-j\alpha}\sup_{\tau<t} \|f(\cdot,t)\|_{C^\alpha(\R^n)} .
\]
We need therefore to show
\beq 
\label{eq:claim}
 \int_{0}^t\|\mathcal{F}^{-1}\big( e^{-(g(\xi,t)-g(\xi,\tau)) |\xi|^{1+s}} \delta(2^{-j}\xi) \big)\|_{L^1} \, d\tau \leq C  2^{-j(1+s)} .
\eeq
Obviously it is enough to prove \eqref{eq:claim} for $j\geq 2$. Recalling the definition of $g(\xi,t)$ and  using \eqref{eq:a0} we have for $\tau <t$
\[
g(\xi,t)-g(\xi,\tau) \geq c_0(t-\tau) \quad \text{and} \quad |\partial^\beta  (g(\xi,t)-g(\xi,\tau))| \leq C_{|\beta|}\frac{t-\tau}{|\xi|^{|\beta|}}.
\]
Using these one may  prove that,  if $\gamma=(\gamma_1,\dots ,\gamma_n)$ is a multi index of length $n$, then  we have
\[
\begin{split}
\big|\pa_{\xi}^\gamma e^{-(g(\xi,t)-g(\xi,\tau)) |\xi|^{1+s}} \big| &\leq C \sum_{k\leq n} (1+ (t- \tau)|\xi|^{1+s})^{k}|\xi|^{-n} e^{-(g(\xi,t)-g(\xi,\tau)) |\xi|^{1+s}}\\
&\leq C (1+ (t- \tau)|\xi|^{1+s})^{n}|\xi|^{-n}  e^{-c_0(t-\tau) |\xi|^{1+s}},
\end{split}
\]
where the constant $C$ depends on the index $n$. Therefore we have by the above and by $\delta(2^{-j}\xi) = 0$ for $\xi \notin B_{2^{j+1}}\setminus B_{2^{j-1}}$ that
\[
|\pa_\xi^\gamma (e^{-(g(\xi,t)-g(\xi,\tau)) |\xi|^{1+s}} \delta(2^{-j}\xi  ))| \leq C (1+(t- \tau) |\xi|^{1+s})^{n}|\xi|^{-n}  e^{-c_0(t-\tau) |\xi|^{1+s}} \tilde \delta(2^{-j}\xi  ),
\]
where $ \tilde \delta$ is a smooth function such that $\tilde \delta(\xi  ) = 0$ for  $\xi \notin B_{2^{j+1}}\setminus B_{2^{j-1}}$.  Using Cauchy-Schwarz  inequality, Plancherel's theorem  and recalling that $\delta(2^{-j}\xi) = 0$ for $\xi \notin B_{2^{j+1}}\setminus B_{2^{j-1}}$ have that
\begin{align*}
     \|\mathcal{F}^{-1}&\big(e^{-(g(\xi,t)-g(\xi,\tau)) |\xi|^{1+s}} \delta(2^{-j}\xi) \big)\|_{L^1}\\
&\leq \left( \int_{\R^n} \frac{1}{(1+|2^jx|^n)^2}dx \right)^{\tfrac12} \left( \int_{\R^n} (1+|2^j x|^n)^2|\mathcal{F}^{-1}\big( e^{-(g(\xi,t)-g(\xi,\tau)) |\xi|^{1+s}} \delta(2^{-j}\xi) \big)|^2 \, d x \right)^{\tfrac12}\\
    &\leq C2^{-\frac{jn}{2}} \left( \sum_{|\gamma|\leq n} \int_{\R^n}\big|2^{jn}\pa_\xi^\gamma(e^{-(g(\xi,t)-g(\xi,\tau)) |\xi|^{1+s}} \delta(2^{-j}\xi )) \big|^2 d\xi \right)^{\tfrac12}\\
& \leq C2^{-\frac{jn}{2}} \left( \int_{B_{2^{j+1}}\setminus B_{2^{j-1}}} |2^{jn} (1+ (t- \tau)|\xi|^{1+s})^{n}|\xi|^{-n} e^{-c_0(t-\tau)  |\xi|^{1+s}}|^2 dx\right)^{\tfrac12}
     \\
    & \leq C (1+(t-\tau) 2^{(j+1)(1+s)})^n e^{-c_0(t-\tau)2^{(j-1)(1+s)}},
\end{align*}
where in the last inequality we used that $\xi \in B_{2^{j+1}}\setminus B_{2^{j-1}}$ and the assumption $j \geq 2$.
Finally we have
\begin{align*}
    \int_0^t \|\mathcal{F}^{-1}\big( e^{-(g(\xi,t)-g(\xi,\tau)) |\xi|^{1+s}}  &\delta(2^{-j}\xi) \big)\|_{L^1}  \, d \tau \leq C \int_0^t (1+  (t-\tau) 2^{(j+1)(1+s)})^n e^{-c_0 (t-\tau) 2^{(j-1)(1+s)}}d\tau
    \\ &\leq C 2^{-(j-1)(1+s)} \int_0^\infty (1+ 2^{2(1+s)}\mu)^n e^{-c_0 \mu} \, d \mu \leq C 2^{-j(1+s)},
\end{align*}
and  we obtain \eqref{eq:claim}.
\end{proof}

We are now ready prove  Theorem \ref{teo:parabolic}. The proof follows from small perturbation argument, where we localize the equation and freeze the coefficients. The argument is similar to the proof
of Theorem \ref{thm2:schauder}, with the difference that the Euclidean space is not compact. We may overcome this by using  Remark \ref{rem:lem2.5-eucl} instead of Lemma  \ref{lem:old.4.3}.  
\begin{proof}[\textbf{Proof of Theorem \ref{teo:parabolic}}]

We begin by noticing that applying the maximum principle, as in the proof of Theorem \ref{thm2:schauder}, we have 
\beq \label{eq:schauder-app}
\sup_{t <T}\|u(\cdot, t)\|_{C^0(\R^n)} \leq C(1+T) \sup_{t <T}\|f (\cdot, t)\|_{C^0(\R^n)}.
\eeq

For the H\"older continuity we  localize the equation. First, fix  $\delta  \in (0,1)$ and choose  $t_0 \in (0,T)$  and  $x_0 \in \R^n$ such that 
\beq \label{eq:choice-x-0}
\sup_{t <T} \sup_{\substack{y\neq x \in \R^n \\|y-x|<\delta/2}}\frac{|\nabla u(y,t) -\nabla u(x,t)|}{|y-x|^{s+\alpha}} \leq 2 \sup_{y \in  B_{\delta/2}(x_0)} \frac{|\nabla u(y,t_0) -\nabla u(x_0,t_0)|}{|y-x_0|^{s+\alpha}}
\eeq
  Let $\eta :\R \to [0,1]$ be a smooth cutoff function such that $\eta(r) =1$ for $|r|\leq \frac12$, $\eta(r) =0 $ for $|r|\geq 1$, denote
\[
\eta_{x_0}(x) = \eta \left(\frac{|x-x_0|}{\delta} \right)
\]
 and let $v= u(x,t)\eta_{x_0}(x)$.  Notice that 
\[
 L_A[v] =\eta_{x_0}  L_A[u]  +\int_{\R^n} u(y,t) \frac{\eta_{x_0}(y)-\eta_{x_0}(x)}{\|y-x\|_{A(x,t)}^{n+1+s}}  \, dy.
\] 
If $u$ is a solution of the equation $\pa_t u= L_A[u] + f(x,t)$, then $v= u\eta_{x_0}$ solves 
\[
\pa_t v= L_{A_{x_0}}[v]  + \tilde f(x,t),
\]
where $A_{x_0} = A(x_0, t)$ and 
\[\begin{split}
\tilde f= \eta_{x_0} f  + \overbrace{(L_{A}-L_{A_{x_0}} )[v]}^{= I_1}-\int_{\R^n} u(y,t) \frac{\eta_{x_0}(y)-\eta_{x_0}(x)}{\|y-x\|_{A(x,t)}^{n+1+s}}  \, dy
= \eta_{x_0} f+I_1- I_2.
\end{split}
\]
 Theorem \ref{thm:timedependent} implies 
\[
\sup_{(0,T)}\|v(\cdot, t)\|_{C^{1+s+\alpha}(\R^n)} \leq C \sup_{(0,T) }\|\tilde f(\cdot, t)\|_{C^{\alpha}(\R^n)}.
\]
We claim that it holds 
\beq \label{eq:claim1}
\sup_{(0,T)}( \| I_1\|_{C^{\alpha}(\R^n)} + \| I_2\|_{C^{\alpha}(\R^n)} ) \leq  C\delta^{\alpha} \sup_{(0,T)}\| v(\cdot, t)\|_{C^{1+s+\alpha}(\R^n)} +  C_\delta\sup_{(0,T)} \|u(\cdot, t)\|_{C^{1+s+\alpha/2}(\R^n)}.
\eeq

Since in the argument to prove \eqref{eq:claim1} the time does not play any role, we  drop it in order to simplify the notation. For $\mu \in [0,1]$ we   denote $A_\mu(x) = (1-\mu)A(x_0) + \mu A(x)$ and write $I_1$ as
\[
\begin{split}
&(L_{A} - L_{A_{x_0}})[v](x) =\int_{\R^n} (v(y)-v(x)) \int_0^1 \frac{d}{d\mu}\left( \frac{1}{\la A_\mu(x) (y-x),(y-x)\ra} \right)^{\frac{n+1+s}{2}} \, d \mu \, dy\\
&=-\tfrac{n+ 1+s}{2}\int_{\R^n}  \int_0^1 \frac{(v(y)-v(x))}{\la A_\mu(x) (y-x),(y-x)\ra^{\frac{n+1+s}{2}}}\left( \frac{\la (A(x) -A(x_0))(y-x),(y-x)\ra}{\la A_\mu(x) (y-x),(y-x)\ra} \right) \, d \mu \, dy\\
&= -\tfrac{n+1+s}{2} \int_{\R^n}  \int_0^1 \frac{(v(y)-v(x) - \eta(|y-x|) \la \nabla v(x), y-x \ra)}{\la A_\mu(x) (y-x),(y-x)\ra^{\frac{n+1+s}{2}}}\left( \frac{\la (A(x) -A(x_0))(y-x),(y-x)\ra}{\la A_\mu(x) (y-x),(y-x)\ra} \right) \, d \mu \, dy,
\end{split}
\]
where the last equality follows from symmetry. Let us define $F: \R^n \times \R^n \to \R$ as
\[
F(y,x) = (v(y)-v(x)- \eta(|y-x|) \la \nabla v(x), y-x \ra) \left( \frac{\la (A(x) -A(x_0))(y-x),(y-x)\ra}{\la A_\mu(x) (y-x),(y-x)\ra} \right),
\]
for $y \neq x$ and $F(x,x) = 0$. Then one may check that $F$ satisfies the conditions in Remark \ref{rem:lem2.5-eucl} with a constant 
\[
\kappa_0 = C\delta^\alpha \|v\|_{C^{1+s+\alpha}} + C_\delta  \|v\|_{C^{1+s+\alpha/2}}.
\]
We leave the details for the reader as it follows from \eqref{eq:g_h-3} and using an argument similar to the one in the proof of Lemma \ref{lem:bound-G-error} . We may thus use Remark \ref{rem:lem2.5-eucl} and obtain
\[
\begin{split}
\| I_1(\cdot, t)\|_{C^{\alpha}(\R^n)} &\leq  C \delta^\alpha \|v(\cdot, t)\|_{C^{1+s+\alpha}(\R^n)} + C_\delta  \|v(\cdot, t)\|_{C^{1+s+\alpha/2}(\R^n)}\\
&\leq   C \delta^\alpha \|v(\cdot, t)\|_{C^{1+s+\alpha}(\R^n)} + C_\delta  \|u(\cdot, t)\|_{C^{1+s+\alpha/2}(\R^n)}.
\end{split}
\]
The argument to bound  the term $I_2$ is similar. We write it as
\[
\begin{split}
I_2 =&\int_{\R^n} u(y) \frac{ (\eta_{x_0}(y) -\eta_{x_0}(x))}{\|y-x\|_{A(x)}^{n+1+s}}\, dy =	\int_{\R^n}\frac{ ( u(y)- u(x)) (\eta_{x_0}(y) -\eta_{x_0}(x))}{\|y-x\|_{A(x)}^{n+1+s}}\, dy \\
&+ u(x) \int_{\R^n}\frac{\eta_{x_0}(y) -\eta_{x_0}(x)- \eta(|y-x|) \la \nabla \eta_{x_0}(x), y-x \ra}{\|y-x\|_{A(x)}^{n+1+s}}\, dy.
\end{split}
\]
We apply Remark \ref{rem:lem2.5-eucl}, by  first choosing 
\[
F_1(y,x) = (u(y)-u(x)) (\eta_{x_0}(y) -\eta_{x_0}(x))
\] 
and then 
\[
F_2(y,x) = \eta_{x_0}(y) -\eta_{x_0}(x)- \eta(|y-x|) \la \nabla \eta_{x_0}(x), y-x \ra
\] 
  to infer
\[
\| I_2(\cdot, t)\|_{C^{\alpha}(\R^n)} \leq  C_\delta  \|u(\cdot, t)\|_{C^{1}(\R^n)}.
\]
Hence we have \eqref{eq:claim1}.

Let us finally show how the claim follows from \eqref{eq:claim1}. First, by choosing $\delta$ small we have 
\[
\sup_{t<T}\|v(\cdot, t)\|_{C^{1+s+\alpha}(\R^n)} \leq C_\delta	 \sup_{t<T}(\|u(\cdot, t)\|_{C^{1+ s+\alpha/2}(\R^n)} + \|f(\cdot, t)\|_{C^\alpha(\R^n)}).
\]
By using $\|u(\cdot, t)\|_{C^{1+s+\alpha}(B_{\delta/2}(x_0))} \leq \|v(\cdot, t)\|_{C^{1+s+\alpha}(\R^n)} $ and by the choice of $x_0$ and $t_0$ in \eqref{eq:choice-x-0} we deduce
\[
\sup_{t<T}\|u(\cdot, t)\|_{C^{1+s+\alpha}(\R^n)} \leq C_\delta \sup_{t<T}(\|u(\cdot, t)\|_{C^{1+ s+\alpha/2}(\R^n)} + \|f(\cdot, t)\|_{C^\alpha(\R^n)}).
\]
By interpolation \eqref{def:interpolation-eucl} we have
\[
\begin{split}
\|u(\cdot, t)\|_{C^{1+ s+\alpha/2}(\R^n)} &\leq C\|u(\cdot, t)\|_{C^{1+ s+\alpha}(\R^n)}^\theta\|u(\cdot, t)\|_{C^{0}(\R^n)}^{1-\theta}\\
&\leq \eps \|u(\cdot, t)\|_{C^{1+ s+\alpha}(\R^n)} + C_\eps\|u(\cdot, t)\|_{C^{0}(\R^n)}.
\end{split}
\]
By choosing $\eps$ small we then have 
\[
\sup_{t<T}\|u(\cdot, t)\|_{C^{1+s+\alpha}(\R^n)} \leq C \sup_{t<T}( \|f(\cdot, t)\|_{C^\alpha(\R^n)}+ \|u(\cdot, t)\|_{C^0(\R^n)}).
\]
The claim then follows from  \eqref{eq:schauder-app}.
\end{proof}

\vspace{4pt}
\noindent

\section*{Acknowledgments}
\noindent
The first author was supported by the Academy of Finland grant 347550.
The second author is a member of GNAMPA.

\end{document}